\def\arxiv{1}
\ifdefined\arxiv\PassOptionsToPackage{pagebackref=true}{hyperref}\fi
\documentclass[pdflatex,sn-mathphys-num]{sn-jnl}

\usepackage{graphicx} %
\usepackage{amsmath,amscd}
\usepackage{changepage}
\usepackage{enumitem}
\usepackage{amsfonts}
\usepackage{amssymb}
\usepackage{comment}
\usepackage{epstopdf}
\usepackage{multirow}
\usepackage{etoolbox}
\usepackage{mathtools,soul}
\usepackage{tikz}
\usepackage{pgfplots}
\pgfplotsset{compat=1.17}
\ifdefined\arxiv
\fi
\usepackage{footnotebackref}
\usepackage{dsfont}
\usepackage[OT1]{fontenc}
\usepackage[framemethod=TikZ]{mdframed}
\usetikzlibrary{arrows.meta,calc,positioning}
\DeclareMathOperator{\prox}{prox}

\DeclareMathOperator{\crit}{crit}
\DeclareMathOperator{\argmin}{argmin}

\newcommand{\inprod}[2]{\langle#1,\,#2\rangle}
\newcommand{\norm}[1]{\|#1\|}
\usepackage[normalem]{ulem}

\usepackage{xcolor}

\newcommand{\sign}[1]{\mathrm{sgn}(#1)}
\def\R{\mathbb{R}}
\newcommand{\M}{\mathcal{M}}
\newcommand{\C}{\mathcal{C}}
\DeclareMathOperator{\dist}{dist}

\usepackage{aliascnt}
\usepackage{bookmark}
\usepackage[capitalize,nameinlink]{cleveref}
\usepackage{footnotebackref}
\theoremstyle{thmstyleone}
\newtheorem{theorem}{Theorem}[section]
\newaliascnt{lemma}{theorem}
\newtheorem{lemma}[lemma]{Lemma}
\aliascntresetthe{lemma}
\newaliascnt{proposition}{theorem}
\newtheorem{proposition}[proposition]{Proposition}
\aliascntresetthe{proposition}
\newaliascnt{corollary}{theorem}
\newtheorem{corollary}[corollary]{Corollary}
\aliascntresetthe{corollary}
\theoremstyle{thmstylethree}
\newaliascnt{definition}{theorem}
\newtheorem{definition}[definition]{Definition}
\aliascntresetthe{definition}
\newtheorem{assumption}{Assumption}
\theoremstyle{thmstyletwo}
\newaliascnt{remark}{theorem}
\newtheorem{remark}[remark]{Remark}
\aliascntresetthe{remark}
\numberwithin{equation}{section}

\crefname{lemma}{Lemma}{Lemmas}
\crefname{proposition}{Proposition}{Propositions}
\crefname{corollary}{Corollary}{Corollaries}
\crefname{definition}{Definition}{Definitions}
\crefname{assumption}{Assumption}{Assumptions}
\crefname{remark}{Remark}{Remarks}
\crefname{section}{Section}{Sections}
\crefname{subsection}{Section}{Sections}
\newcommand{\crefhead}[2]{\texorpdfstring{\cref{#1}}{#2~\ref{#1}}}
\Crefname{section}{Section}{Sections}
\crefformat{equation}{\textup{#2(#1)#3}}
\crefrangeformat{equation}{\textup{#3(#1)#4--#5(#2)#6}}
\crefmultiformat{equation}{\textup{#2(#1)#3}}{ and \textup{#2(#1)#3}}
{, \textup{#2(#1)#3}}{, and \textup{#2(#1)#3}}
\crefrangemultiformat{equation}{\textup{#3(#1)#4--#5(#2)#6}}%
{ and \textup{#3(#1)#4--#5(#2)#6}}{, \textup{#3(#1)#4--#5(#2)#6}}{, and \textup{#3(#1)#4--#5(#2)#6}}

\Crefformat{equation}{#2Equation~\textup{(#1)}#3}
\Crefrangeformat{equation}{Equations~\textup{#3(#1)#4--#5(#2)#6}}
\Crefmultiformat{equation}{Equations~\textup{#2(#1)#3}}{ and \textup{#2(#1)#3}}
{, \textup{#2(#1)#3}}{, and \textup{#2(#1)#3}}
\Crefrangemultiformat{equation}{Equations~\textup{#3(#1)#4--#5(#2)#6}}%
{ and \textup{#3(#1)#4--#5(#2)#6}}{, \textup{#3(#1)#4--#5(#2)#6}}{, and \textup{#3(#1)#4--#5(#2)#6}}

\crefdefaultlabelformat{#2\textup{#1}#3}

\hypersetup{
  hidelinks,
  pdftitle={Transferability of Error Bounds and the Kurdyka-\L ojasiewicz Property under C1 Partial Smoothness},
  pdfauthor={Hanju Wu, Ching-pei Lee and Yue Xie}
}
\begin{document}

\title[Transferability of EB and KL under $C^1$ partial smoothness]{Transferability of Error Bounds and the Kurdyka-\L ojasiewicz Property under $C^1$ Partial Smoothness}

\author[1]{\fnm{Hanju} \sur{Wu}} \email{wuhanju@connect.hku.hk}
\author[2,4,5]{\fnm{Ching-pei} \sur{Lee}} \email{leechingpei@gmail.com}
\author*[3]{\fnm{Yue} \sur{Xie}}\email{yxie21@hku.hk}

\affil[1]{\orgdiv{Department of Mathematics}, \orgname{The University of Hong Kong}, \orgaddress{\city{Pokfulam}, \country{Hong Kong}}}
\affil[2]{\orgname{Institute of Statistical Mathematics}, \orgaddress{\city{Tokyo}, \country{Japan}}}
\affil[3]{\orgdiv{Department of Mathematics and Institute of Data Science}, \orgname{The University of Hong Kong}, \orgaddress{\city{Pokfulam}, \country{Hong Kong}}}
\affil[4]{\orgname{University of Tokyo}, \orgaddress{\city{Tokyo}, \country{Japan}}}
\affil[5]{\orgname{Graduate University for Advanced Studies, SOKENDAI}, \orgaddress{\city{Tokyo}, \country{Japan}}}

\abstract{This work studies how a generalized error bound (GEB) property in the
ambient Euclidean space transfers to the active manifold under $\C^1$
partial smoothness.
The GEB generalizes the standard error bound, which upper bounds
the distance to a subset of critical points using first-order
stationarity residuals, by allowing the distance to be composed with a
gauge function.
In the ambient space, these residuals are measured by the
subdifferential or the slope of the original function,
while on the manifold they are measured by the Riemannian gradient or
the slope of the restricted function.
We develop complementary geometric and metric analyses to show
that the ambient and manifold GEBs are equivalent with the same gauge
function.
The geometric argument shows that the projection of the
subdifferential onto the tangent space coincides with the Riemannian gradient,
while the metric approach combines slope estimates, identifiability,
and a new uniform linear-growth result for $\C^1$ manifolds.
We further extend our analysis to establish the analogous equivalence
for the Kurdyka--\L ojasiewicz (KL) property, with the same
desingularizer, and separately characterize their connections with the GEB.
Finally, for regularized optimization, we relate subdifferential and
proximal error bounds,
thereby placing the previously established equivalence between
proximal and manifold error bounds for $\ell_1$-regularization within
a broader framework.}

\keywords{error bounds, identifiability, Kurdyka--\L ojasiewicz property, partial smoothness, linear-growth}

\pacs[MSC Classification]{49J52, 49J53, 90C30, 58C05, 58C06}

\maketitle

\section{Introduction}
\label{sec:intro}

Active manifolds capture the local structure of many nonsmooth
optimization problems, such as the support of a sparse solution
and the rank of a matrix solution. On an active manifold, the nonsmooth
function inducing the structure often becomes smooth. This viewpoint
is formalized through \emph{partial smoothness}
\cite{lewis2002active,drusvyatskiy2012optimalityidentifiabilitysensitivity}
and is closely related to
\emph{identifiability}
\cite{wright1993identifiable,hare2004identifying,drusvyatskiy2014optimality},
which characterizes when a sequence approaching a critical point
eventually remains on the active manifold.
This framework supports sensitivity analysis and active-set
algorithms, and enables local analysis of a nonsmooth problem through
its smooth restriction.

To leverage this smooth reduction in convergence analysis, one needs to understand
how quantitative regularity in the ambient space transfers to the
active manifold.
Two prominent regularity properties, the Kurdyka--\L ojasiewicz (KL) property and the error bound (EB),
have been central to such analyses; see, for example,
\cite{luo1993error,tseng2009coordinate,tseng2010approximation,yue2019family,AttBS13a,bolte2014proximal}.
The KL property relates a first-order stationarity residual to the
objective-value gap and uses a desingularizer to quantify the
relationship, whereas an EB upper bounds the distance to a solution or
critical-point set by this residual.
Motivated in part by the flexibility of the desingularizer in the KL
property, we work with a generalized error bound (GEB) that replaces
the linear dependence on the distance in the standard EB by a possibly
nonlinear gauge function.
The desingularizers and gauge functions allow us to compare the
quantitative strength of the corresponding KL and GEB properties.

In \cite{wu2025resolution}, %
we showed that for $\ell_1$-regularized problems, the standard
proximal EB at a nondegenerate minimizer transfers to the standard EB
on the active manifold determined by its support.\footnote{Throughout,
	by \emph{transfer} from one side to the other,
	we mean that when the regularity condition holds on the former,
	it also holds on the latter with the same gauge function or
	desingularizer, and a \emph{sharper} gauge or desingularizer is
one with greater quantitative strength.}
This result concerns only the linear gauge on both sides, but
intuitively, one might expect that restriction to a smaller set
could sharpen regularity.
This motivates our questions:

\smallskip
\begin{mdframed}[innerleftmargin=4pt,innerrightmargin=4pt]
	\centering \it
	Does the same transfer
	extend beyond $\ell_1$ regularization to general partly smooth
	problems, and does restriction to the active manifold yield a
	sharper error bound?
\end{mdframed}

Existing analyses provide mixed responses to these questions.
On one hand,
a smooth loss can have a singular positive semidefinite
Hessian in the ambient space, while its restriction to the tangent
space  of the active manifold is
positive definite.
To illustrate this, consider
\begin{equation*}
  f(x)=g(x)+|x_1|+|x_2|,\qquad
  g(x)=0.5 (x_1+0.5 x_2-2)^2.
\end{equation*}
At the nondegenerate minimizer $\bar x=(1,0)$, the active manifold is
$\mathcal M=\{x_2=0\}$, and
\[
  \nabla^2g \equiv
  \begin{pmatrix}1&0.5 \\[2pt]0.5 &0.25 \end{pmatrix},
  \qquad
  f(x_1,0)=1.5 +0.5(x_1-1)^2, \quad(x_1>0).
\]
The ambient Hessian of $g$ has the null direction $(1,-2)$, whereas
the second derivative of $f|_{\mathcal M}$ is $1$ near $\bar x$.
Such a regularity condition on the active manifold has been used to
design algorithms with provably fast local convergence
\cite{bertsekas1982projected,gafni1984two},
and to explain the empirically observed fast local convergence of existing methods
\cite{liang2017activity}.
Similarly, \cite{Lee2023a,hu2025analysis} assume
on-manifold error-bound conditions to design their algorithms and
establish local superlinear convergence.

On the other hand, existing transfer results for other regularity
properties suggest otherwise.
Under suitable assumptions, local minimality, strict
local minimality, and quadratic growth are preserved upon
restriction to an active manifold
\cite{lewis2002active,lewis2013partial,drusvyatskiy2012optimalityidentifiabilitysensitivity,lewis2024identifiability}.
In the other direction,
\cite{wright1993identifiable,hare2004identifying,wright2012accelerated}
showed that quadratic growth propagates from the active manifold
to the ambient space in various settings, and
\cite{li2018calculus} established manifold-to-ambient transfer of the
KL property under partial smoothness.
Lewis and Tian \cite{lewis2024identifiability} subsequently studied the
equivalence of the KL property with general desingularizers on
identifiable sets in complete metric spaces.
As will be discussed in \cref{sec:kl}, the
ambient-to-identifiable-set implication requires additional care and
fails in this generality, which provides an additional motivation for
our Euclidean analysis under $\C^1$ partial smoothness.
Nevertheless, the manifold-to-ambient transfers above rule out
sharpening under their respective hypotheses, but leave the question
open for error bounds.

To resolve this tension, we aim for minimal assumptions
and study the transferability of GEBs between the ambient space and
the active manifold.
We therefore only assume that the underlying active manifold is
$\C^1$, the minimal regularity needed for the tangent space and the
Riemannian gradient to be well defined.
Such manifolds arise, for example, in the following
$\ell_p$-norm constrained optimization problem with $p \in
(1,2)$:
\begin{equation}
	  \min_{x\in\R^n} g(x) \quad\text{s.t.}\quad \|x\|_p\le r.
  \label{eq:intro-lp-example}
\end{equation}
The active manifold at any $\bar x$ with $\|\bar x\|_p=r$ is $\M = \{x
\in \R^n: \|x\|_p = r\}$,
and $F(x) = \sum_{i=1}^n |x_i|^p - r^p$ is a local defining map for $\M$
on $\R^n \setminus \{0\}$, where
$\nabla F(x) = p(\sign{x_i}|x_i|^{p-1})_{i=1}^n \neq 0$
is H\"older continuous with exponent $p-1$, so $\M$ is $\C^{1,p-1}$ but not
$\C^2$, and this regularity is sharp at points of $\M$ with zero
components.

Analyses in the literature regarding partly smooth functions often assume
$\C^2$ smoothness so that the nearest-point projection onto the
manifold is single valued and well defined.
Without these properties, we develop new tools by constructing a
retraction-like mapping through a local graph view of $\C^1$ manifolds.
Our results show that, under such minimal assumptions (together with
nondegeneracy), the ambient and manifold GEBs transfer to each other,
and we further confirm the same for the KL property.
Consequently, the answers to the questions raised above are affirmative
for the first but negative for the second:
\emph{Restriction to the active manifold does not sharpen the local EB or KL
property} in this sense, although the constants and neighborhoods may
differ.
This negative result also suggests that the fast local convergence
results in existing works might hold regardless of
identifying the active manifold, which is consistent with and provides
a complementary explanation to the recent discovery of
\cite{diaz2026active} for degenerate situations.
On the positive side, our ambient-to-manifold transfer result turns
the extensive literature on ambient error bounds, including
\cite{luo1993error,tseng2010approximation,hou2013linear,zhou2017unified},
into sufficient conditions for on-manifold GEBs that guarantee fast local
convergence of methods like those in \cite{Lee2023a,hu2025analysis}.

More detailed contributions of this work are summarized below.
\begin{itemize}[leftmargin=3em]
  \item We derive a uniform linear-growth estimate for nearby points
	  off the active manifold under $\C^1$ partial smoothness and
	  nondegeneracy.
	  The estimate requires only a $\C^1$ active manifold and a $\C^1$
	  restriction, and supplies the local control needed for the
	  transfer arguments without relying on a smooth nearest-point
	  projection available for $\C^2$ manifolds
	  (\cref{thm:linear-growth}).

  \item We establish the equivalence between ambient and manifold GEBs
  with the same gauge function under $\C^1$ partial smoothness and
  nondegeneracy, relative to subsets of critical points on which $f$
  is constant near a reference point.
  We give two complementary proofs:
  an intuitive geometric proof relating subdifferential residuals
  to the Riemannian gradient (\cref{thm:eb-equiv-geometric}),
  and a slope- and identifiability-based metric proof
  (\cref{thm:eb-equiv-metric}) that has the potential to extend to
  more general settings.
  For $\ell_1$-regularized problems, this result recovers the known
  transferability of the error bound as a special case
  (\cref{sec:l1-recovery}).

  \item We show ambient-to-manifold transfer of the KL property under
	  $\C^1$ partial smoothness and nondegeneracy for concave desingularizers
	  (\cref{thm:kl-ambient-to-manifold}).
  Together with the known converse, this yields equivalence with
  the same desingularizer in the Euclidean setting.
  We also provide
  a counterexample in $\R^2$
  (\cref{sec:counterexample})
  showing that the ambient-to-restricted implication may fail for a
  general identifiable set.

  \item We give sufficient conditions for a GEB to imply a KL
  inequality and quantify the relationship between its gauge function
  and the resulting desingularizer (\cref{prop:eb-kl}). This implication does not
  require partial smoothness and is therefore of independent
  interest.
\end{itemize}

This paper is organized as follows. \Cref{sec:preliminary}
presents the preliminaries, and \cref{sec:identify} develops
the tools based on partial smoothness, identifiability, and linear
growth.
The GEB and KL transfer results are established in
\cref{sec:eb,sec:kl}.
In \cref{sec:discussion-application}, we discuss regularized
optimization and apply our results to the motivating examples of
$\ell_1$-regularization and $\ell_p$-norm constraints.
Some final comments appear in \cref{sec:conclusion}.
The appendices contain the
$\R^2$ counterexample and the omitted proofs.

\section{Notation and Preliminaries}
\label{sec:preliminary}
In this paper, $\langle\cdot,\cdot\rangle$ denotes the Euclidean inner
product on $\mathbb{R}^n$, $\|\cdot\|$ the induced Euclidean norm, and
$d$ the Euclidean distance between two points, namely
$d(x,y)=\|x-y\|$.
For a set $C\subset\mathbb{R}^n$, $\operatorname{ri}C$ denotes its
relative interior, $\operatorname{aff}C$ its affine hull, and
$\operatorname{par}C$ the subspace of $\mathbb{R}^n$ parallel to
$\operatorname{aff}C$.
We write $\operatorname{dist}(x,C)\coloneqq\inf_{y\in C}\|x-y\|$ as
the distance between $x$ and $C$,
and follow the convention to set
$\operatorname{dist}(x,\emptyset)=+\infty$.
We use $P_C(x)$ for the Euclidean
projection onto $C$ whenever it is well defined and single valued.
For any given point $x$, the Euclidean open ball centered at $x$ with
radius $\epsilon$ is denoted by $B_\epsilon(x)$.
The unit sphere is denoted by $\mathbb{S}^{n-1}\coloneqq \{v \in
	\mathbb{R}^n : \|v\|=1\}$.
The extended real number is defined as $\overline{\R}\coloneqq
\R\cup\{+\infty\}$, and the nonnegative real number is denoted by
$\R_+$.
For an extended-real-valued function $f:\R^n \to \overline{\R}$, its
domain is denoted by $\operatorname{dom}f \coloneqq \{x \in
	\mathbb{R}^n : f(x) < +\infty\}$.
The sublevel set of $f$ at level $\alpha$ is denoted by $[f\le \alpha] \coloneqq \{x \in \mathbb{R}^n : f(x)\le \alpha\}$,
and we also use $[\alpha \le f \le \beta]$ to denote the set
$\{x\in\mathbb{R}^n:\alpha\le f(x)\le \beta\}$.
We write $x_k\to_f x$ for $(x_k,f(x_k))\to(x,f(x))$.
We let $\operatorname{Id}$ denote the identity map.

\subsection{Manifolds}
We follow the standard definition and notation for manifold-related
concepts in \cite{boumal2023intromanifolds}.
All manifolds $\mathcal{M}$ are assumed throughout to be embedded in $\mathbb{R}^n$ and we consider the tangent and normal spaces to $\mathcal{M}$ as subspaces of $\mathbb{R}^n$.
A non-empty set $\mathcal{M} \subset \mathbb{R}^{n}$ is a
$\C^{p}$-smooth embedded submanifold of dimension $m$ if around any $x
\in \mathcal{M}$ there exists an open neighborhood $U \subset
\mathbb{R}^{n}$ of $x$ and a $\C^{p}$-smooth map $F: U \rightarrow
\mathbb{R}^{n-m}$ such that $\mathcal{M} \cap U=F^{-1}(0)$, and for
all $y \in U$, the Jacobian matrix $\nabla F(y)$ has full rank $n-m$.
We call such an $F$ a \emph{local defining map} for $\mathcal{M}$.
The tangent space of $\mathcal{M}$ at $x$ is defined as $T_{x} \mathcal{M}\coloneqq\ker \nabla F(x)$.
Another definition of the tangent space is that $T_{x} \mathcal{M}$
consists of all tangent vectors to smooth curves on $\mathcal{M}$
passing through $x$, i.e., $\dot{\gamma}(0)$ for all $\C^{1}$-smooth
curves $\gamma:(-\epsilon, \epsilon) \rightarrow \mathcal{M}$ with
$\gamma(0)=x$. The normal space $N_x \M$ of $\mathcal{M}$ at $x$ is defined as the orthogonal complement of $T_{x} \mathcal{M}$, i.e., $N_{x} \mathcal{M}\coloneqq\left(T_{x} \mathcal{M}\right)^{\perp}=\operatorname{Range}\left(\nabla F(x)^\top\right)$.

When $\mathcal{M}$ is $\C^{2}$-smooth, the nearest-point projection $P_{\mathcal M}$ is well defined,
single valued, and $\C^1$-smooth in a neighborhood of
$\mathcal M$, and thus $\M$ being $\C^2$-smooth is the standard
assumption in the literature for partial smoothness or
identifiability (c.f. \cite{drusvyatskiy2012optimalityidentifiabilitysensitivity,lewis2022partial,lewis2024identifiability}).
On the contrary, this work tries to relax the $\C^2$-smoothness
condition to $\C^1$ for accommodating more general applications like
those we discussed in \cref{sec:intro}.
In this scenario, projection may not be well defined or smooth, and we
therefore develop new tools that do not resort to the projection.

For a $\C^{p}$-smooth manifold $\mathcal{M}$, a function $f: \mathcal{M} \rightarrow \mathbb{R}$ is said to be $\C^{p}$-smooth if for each $x \in \mathcal{M}$ there exists an open neighborhood $U \subset \mathbb{R}^{n}$ of $x$ and a $\C^{p}$-smooth function $\tilde{f}: U \rightarrow \mathbb{R}$ such that $\tilde{f}|_{\mathcal{M} \cap U}=f|_{\mathcal{M} \cap U}$. The function $\tilde{f}$ is called a $\C^{p}$-smooth extension of $f$ around $x$.

It is convenient to equip each tangent space of the manifold
$\mathcal{M}$ with an inner product $\langle\cdot, \cdot\rangle_{x} :
T_{x} \mathcal{M} \times T_{x} \mathcal{M} \rightarrow \mathbb{R}$. If
this inner product varies smoothly with $x$, then $\mathcal{M}$ is
called a Riemannian manifold and the inner product is called a
Riemannian metric. When $\mathcal{M}$ is an embedded submanifold of
$\mathbb{R}^{n}$, the tangent space $T_{x} \mathcal{M}$ is a subspace
of $\mathbb{R}^{n}$.
It is therefore natural to choose the restriction of the Euclidean
inner product on $\mathbb{R}^{n}$ to $T_{x} \mathcal{M}$ as the Riemannian metric, i.e.,
$\langle\eta, \xi\rangle_{x}\coloneqq\inprod{\eta}{\xi}$ for all $\eta, \xi \in T_{x} \mathcal{M}$.
In this case, we call $\mathcal{M}$ a Riemannian submanifold of $\mathbb{R}^{n}$,
and consequently, the distance measure in this work is also the ambient Euclidean distance.

For a Riemannian manifold $\mathcal{M}$, we can define the Riemannian gradient of a smooth function $f: \mathcal{M} \rightarrow \mathbb{R}$ as follows. The Riemannian gradient $\nabla_{\mathcal{M}} f(x) \in T_{x} \mathcal{M}$ is the unique tangent vector satisfying
\begin{equation*}
    \langle\nabla_{\mathcal{M}} f(x), \xi\rangle_{x}=\mathrm{D}
	f(x)[\xi] = \frac{d}{dt} f(\gamma(t))\big|_{t=0}, \quad \forall
	\xi \in T_{x} \mathcal{M},
\end{equation*}
where $\mathrm{D} f(x)[\xi]$ is the directional derivative of $f$ at $x$ along $\xi$ and $\gamma:(-\epsilon, \epsilon) \rightarrow \mathcal{M}$ is any smooth curve satisfying $\gamma(0)=x$ and $\dot{\gamma}(0)=\xi$.
If $\mathcal{M}$ is a Riemannian submanifold of $\mathbb{R}^{n}$, then the Riemannian gradient can be computed via the projection of the Euclidean gradient onto the tangent space, i.e.,
\begin{equation*}
    \nabla_{\mathcal{M}} f(x)=P_{T_{x} \mathcal{M}} \nabla \tilde{f}(x)
\end{equation*}
where $\tilde{f}: U \rightarrow \mathbb{R}$ is any smooth extension of $f$ around $x$.
We slightly abuse the notation and write $\nabla_{\M} f(x)$ for both $f:
\R^n \to \overline{\R}$ and $f|_{\M}$ as well when $f$ is smooth
on $\M$, even if $f$ is not smooth on $\R^n$.

The following result from \cite[Thm.~2.2, Cor.~2.3]{miller2005newton}
is critical for our analysis.
Its proof is a simple application of the inverse function theorem, and
we include it in \cref{proofs} for completeness.
More details about the inverse function theorem for manifolds can be
found in \cite[Thm.~4.5]{lee2013smooth}.
\begin{lemma}[Manifold locally as a graph]
    \label{lem:local-graph}
    Let $\mathcal{M}\subset \mathbb{R}^n$ be a $\C^p$-smooth embedded
	submanifold around $\bar x\in\mathcal{M}$ for some $p\geq 1$.
    Then,
	there exist a neighborhood $U$ of $\bar x$, $\rho>0$ and a $\C^p$ map
    \[
        h:T_{\bar x}\mathcal{M}\cap B_\rho(0)\to N_{\bar x}\mathcal{M}
    \]
    such that $h(0)=0$, $\nabla h(0)=0$, and
    \[
        \mathcal{M}\cap U
        =
        \left\{\bar x+u+h(u):u\in T_{\bar x}\mathcal{M}\cap
	B_\rho(0)\right\}.
    \]
	Consequently, for all $y$ sufficiently close to $\bar x$ with
	decomposition $y-\bar x=u+v$ for some $u\in T_{\bar x}\mathcal{M}$
	and $v\in N_{\bar x}\mathcal{M}$,
    \[
        R_{\bar x}(y) \coloneqq \bar x+u+h(u)\in\mathcal{M}
    \]
    is well defined, and $R_{\bar x}= \operatorname{Id}$ on $\M$ around $\bar x$.
\end{lemma}

A point $y$ near $\bar x$ admits a unique
decomposition $y=\bar x+u+w$ with $u\in T_{\bar x}\mathcal{M}$ and
$w\in N_{\bar x}\mathcal{M}$, and the retraction-like $R_{\bar x}$ maps $y$ to $\bar
x+u+h(u)\in\mathcal{M}$. Moreover, the curve $c(t)\coloneqq \bar x+tu+h(tu)$ satisfies $c(0)=\bar x$ and $c^{\prime}(0)=u$.

\subsection{Slope and Variational Analysis}
We use standard definitions and notation of variational analysis following \cite{rockafellar2009variational}.
For an extended-real-valued function $f:\mathbb{R}^n\to\overline{\R}$,
$\hat\partial f(\bar x)$ denotes its \emph{Fr\'echet subdifferential}
at $\bar x\in\operatorname{dom}f$, $\partial f(\bar x)$ the
\emph{limiting subdifferential} and $\partial_{P} f(\bar x)$ the \emph{proximal subdifferential},
and it is well-known that
	$\partial_{P} f(\bar x)\subseteq \hat\partial f(\bar x)\subseteq
	\partial f(\bar x)$.
The \emph{limiting critical-point set} is defined as
\begin{equation*}
    \operatorname{crit}f
    \coloneqq \{x\in\operatorname{dom}f:0\in\partial f(x)\}.
\end{equation*}
We say $f$ is \emph{subdifferentially regular} at $\bar x$ if
$\hat{\partial} f(\bar x) = \partial f(\bar x)$, and \emph{prox-regular} at $\bar x$ for $\bar{v} \in \partial f(\bar x)$ if there exist $\epsilon > 0$ and $\rho \ge 0$ such that
\begin{equation}
	\begin{aligned}
		f(x') &\ge f(x) + \langle v, x'-x \rangle
		- \frac{\rho}{2}\|x'-x\|^2,\\
		&\quad \forall x' \in B_{\epsilon}(\bar x),
	    \quad x \in B_{\epsilon}(\bar x)
		\cap [f < f(\bar x) + \epsilon],\quad
		v \in \partial f(x) \cap B_{\epsilon}(\bar v).
	\end{aligned}
	\label{eq:prox-reg-def}
\end{equation}
If $f$ is prox-regular at $\bar x$ for all $\bar{v} \in \partial
f(\bar x)$, we simply say it is prox-regular at $\bar x$.
Prox-regularity implies that for all $(x,v) \in \operatorname{gph}
\partial f$ close enough to $(\bar x, \bar v)$ with $f(x)$ close enough
to $f(\bar x)$, $v$ is a proximal subgradient of $f$ at $x$.

The following proposition shows that the Riemannian gradient is the projection of any Fr\'echet subgradient onto the tangent space of the manifold.
Similar results for $\C^2$ manifolds can be found in \cite[Prop.~2.2]{lewis2002active} and \cite[Prop.~12]{daniilidis2006geometrical}.
We provide the proof of this result in \cref{proofs} for completeness.
\begin{proposition}
    \label{prop:subdiff-inclusion}
    Suppose that $\mathcal{M}$ is a $\C^1$-smooth manifold containing $x$,
	and that
    for a given function $f$, the restriction $f|_{\mathcal{M}}$ is $\C^1$-smooth around $x$.
    Then,
    \begin{equation}
        \operatorname{par}\hat\partial f(x)\subseteq N_x\mathcal M,
        \label{eq:sharpness-easy-inclusion}
    \end{equation}
    and for every $g\in\hat\partial f(x)$, the Riemannian gradient
    of $f|_{\mathcal M}$ satisfies
    \begin{equation}
        \label{eq:Rgrad-Fsubdiff}
        \nabla_{\mathcal{M}} f(x)=P_{T_x\mathcal M}(g).
    \end{equation}
\end{proposition}

We now introduce the \emph{slope} $|\nabla f|(x)$, a central tool in
this paper.
We follow the definition in
\cite[Def.~2.1]{lewis2024identifiability},
and it is easy to see that this definition is equivalent to the
classical one in, for example,
\cite[Def.~2.2.1]{drusvyatskiy2013slope} and
\cite[Def.~3.2]{ioffe2017variational}.
\begin{definition}[Slope; {\cite[Def.~2.1]{lewis2024identifiability}}]
    For a function $f: \mathbb{R}^n \to \overline{\R}$, the slope of $f$ at $\bar x \in \mathbb{R}^n$ is defined as
    \begin{equation*}
        |\nabla f|(\bar x) \coloneqq
		\begin{cases}
			\infty & \text{if } \bar x \notin \operatorname{dom}f,\\
			0 & \text{if $\bar x$ is a local minimizer of $f$},\\
		\limsup_{\substack{x\to \bar x,\ x \ne \bar x}} \frac{f(\bar
		x)-f(x)}{d(\bar x, x)}  & \text{otherwise.}
	\end{cases}
    \end{equation*}
\end{definition}

We recall several basic properties of the slope from \cite[Prop.~8.3,
8.5]{ioffe2017variational}, and \cite[Prop.~2.3.3]{drusvyatskiy2013slope}.
\begin{proposition}
    \label{prop:slope}
    For any closed function $f : \mathbb{R}^n \to \overline{\R}$, we
	always have
    \begin{itemize}
	\item $0 \leq \operatorname{dist} (0, \partial f(x))
		\leq |\nabla f|(x) \le
		\operatorname{dist}(0, \hat\partial f(x))$ for all $x$.
    \item If $f$ is subdifferentially regular at $x$, then $|\nabla
		f|(x) = \operatorname{dist}(0, \partial f(x))$.
    \end{itemize}
\end{proposition}

The following lemma from \cite[Ex.~2.1]{lewis2025complexity} that
relates the quantities $|\nabla (f|_{\mathcal{M}})|(x)$ and
$\|\nabla_{\mathcal{M}} f(x)\|$ will be useful in our analysis.
For completeness, we provide a concise proof in \cref{proofs}.
\begin{lemma}\label{lem:slope-grad}
Suppose that $\mathcal{M}$ is a $\C^1$-smooth embedded submanifold around $x$ and that $f|_{\mathcal{M}}$ is $\C^1$-smooth around $x$, then
$
    |\nabla (f|_{\mathcal{M}})|(x) = \|\nabla_{\mathcal{M}} f(x)\|.
$
\end{lemma}

\section{Partial Smoothness, Identifiability and Linear-Growth}
\label{sec:identify}
This section develops tools essential to the main results of this paper.
We first introduce the definitions of partial smoothness and identifiability.
The following is the definition of partial smoothness from
\cite{drusvyatskiy2012optimalityidentifiabilitysensitivity,lewis2024identifiability}.\footnote{
	Our results need inner semicontinuity only of
	the Fr\'echet subdifferential $\hat\partial f$ rather than of the
	limiting subdifferential $\partial f$ in the
	standard definition of partial smoothness stated here.}
\begin{definition}[$\C^p$-Partial smoothness]
    \label{def:partial-smoothness}
Given a closed function $f \colon \mathbb{R}^n \to
\overline{\R}$, a point $\bar x \in \operatorname{dom}\, f$, and a
$\C^p$-smooth manifold $\mathcal{M}$ containing $\bar x$ for some integer $p \geq
1$. $f$ is \emph{$\C^p$ partly smooth} at $\bar x$ for a subgradient 
$\bar y \in \partial f(\bar x)$ relative to $\mathcal{M}$ if it satisfies the following properties:
\begin{adjustwidth}{.5cm}{0cm}
	\begin{description}
\item[Prox-regularity:]  the function $f$ is prox-regular at $\bar x$
	for $\bar y$.\footnote{This definition differs slightly from the
		original one in
\cite{lewis2002active}, which uses subdifferential regularity instead
of prox-regularity (and treats prox-regularity as an additional,
independent assumption).}
\item[Restricted smoothness:]  the restriction $f|_{\mathcal{M}}$ is $\C^p$-smooth around $\bar x$.
\item[Sharpness:] the subspace $\operatorname{par}\, \hat\partial f(\bar x)$ is just the normal space $N_{\bar x}\mathcal{M} $.
\item[Inner semicontinuity:] for all $y \in \partial f(\bar x)$ near $\bar y$ and every
sequence $x_r \to \bar x$ in $\mathcal{M}$, there exists a sequence $y_r \in \partial f(x_r)$ converging to $y$.
\end{description}
\end{adjustwidth}
\end{definition}
This manifold $\M$ is often called the \emph{active manifold} of $f$ at
$\bar x$.

Notice that $\hat\partial f(x)$ is nonempty for every
$x\in\mathcal M$ sufficiently close to $\bar x$ 
if the function $f$ is $\C^p$-partly smooth at the point $\bar x$ for $0$ relative to $\mathcal{M}$.
Otherwise, there would be a sequence $x_r\to\bar x$ in $\mathcal M$ with
    $\hat\partial f(x_r)=\emptyset$.  Inner semicontinuity at the
    subgradient $0$ then provides $v_r\in\partial f(x_r)$ with
    $v_r\to0$.  Since $f(x_r)\to f(\bar x)$, prox-regularity at
    $\bar x$ for $0$ would imply $v_r\in\partial_P f(x_r) \subseteq \hat\partial f(x_r)$ for all
    large $r$, a contradiction.

For a partly smooth function, the sharpness condition in
\cref{def:partial-smoothness} is stable: it also holds at all nearby
points on the active manifold.
A similar result has been established in
\cite[Prop.~2.10]{lewis2002active} for $\C^2$-partly smooth
functions with a slightly different definition.
\begin{proposition}
    \label{prop:sharpness-stability}
    If the function $f$ is $\C^p$-partly smooth for some $p \geq 1$ at
	the point $\bar x$ for $0$ relative to $\mathcal{M}$, then all
	$x \in \mathcal{M}$ sufficiently close to $\bar x$
satisfy
\begin{equation}
    \operatorname{par}\, \hat\partial f(x) = N_{x}\mathcal{M} 
	\label{eq:sharpness-stability}
\end{equation}
\end{proposition}
\begin{proof}
    From \cref{eq:sharpness-easy-inclusion}, we always have
    $\operatorname{par}\, \hat\partial f(x) \subseteq N_{x}\mathcal{M}$.
    Suppose now that the reverse inclusion fails,
    then there exist points
    $x_r\in\mathcal M$ converging to $\bar x$ and unit vectors
    $
        z_r\in N_{x_r}\mathcal M
        \cap\bigl(\operatorname{par}\hat\partial f(x_r)\bigr)^\perp.
    $
    Passing to a subsequence and using the continuity of the normal
    spaces of a $\C^1$ manifold, we may assume that
    $z_r\to z\in N_{\bar x}\mathcal M$ with $\|z\|=1$.

    Prox-regularity at $\bar x$ for $0$ implies
    $0\in\partial_P f(\bar x)\subseteq\hat\partial f(\bar x)$.
    Let $u,v\in\hat\partial f(\bar x)$ be sufficiently close to $0$
	such that \cref{eq:prox-reg-def} holds.
    By the inner semicontinuity condition in
    \cref{def:partial-smoothness}, there exist
    $u_r,v_r\in\partial f(x_r)$ such that $u_r\to u$ and $v_r\to v$.
    Since $f|_{\mathcal M}$ is continuous, we also have
    $f(x_r)\to f(\bar x)$.  Therefore, by prox-regularity at
	$\bar x$ for $0$, we get that for all sufficiently large $r$,
	$u_r, v_r$ are proximal subgradients, and hence belong to
	$\hat\partial f(x_r)$.
	Consequently,
    $
        \langle z_r,u_r-v_r\rangle=0.
    $
	Letting $r\to\infty$ then yields 
	\begin{equation}
		\langle z,u-v\rangle=0, \quad \forall u,v \in \hat{\partial}
		f(\bar x) \text{ sufficiently close to } 0.
		\label{eq:orth}
	\end{equation}

    Finally, because $\hat\partial f(\bar x)$ is convex and contains
    $0$, its intersection with any neighborhood of $0$ has the same
	parallel space as $\hat\partial f(\bar x)$,  so \cref{eq:orth}
	implies
    $
        z\in\bigl(\operatorname{par}\hat\partial f(\bar x)\bigr)^\perp.
    $
    The sharpness condition at $\bar x$ gives
    $\operatorname{par}\hat\partial f(\bar x)=N_{\bar x}\mathcal M$.
    This contradicts the facts that $z\in N_{\bar x}\mathcal M$ and
	$\|z\|=1$.  Hence \cref{eq:sharpness-stability} holds for all
	$x\in\M$ sufficiently close to $\bar x$.
\end{proof}

Partial smoothness is closely related to the concept of identifiability.
Below we present this concept following the definition from
\cite{drusvyatskiy2014optimality,lewis2024identifiability}.
\begin{definition}[Identifiability; Modulus of Identifiability]
    \label{def:identifiable-subdifferential}
    For a closed function $f : \mathbb{R}^n \to \overline{\R}$, consider a point $\bar x \in \operatorname{dom}f$ with $0 \in \hat{\partial} f(\bar x)$ and a set $\mathcal{M} \subset \mathbb{R}^n$ containing $\bar x$.
        $\mathcal{M}$ is identifiable at $\bar x$ if
    \begin{equation*}
        \eta \coloneqq \liminf_{x \rightarrow _f \bar x, x \notin \mathcal{M}} \operatorname{dist}(0, \partial f(x)) > 0.
    \end{equation*}
	The value $\eta$ is called the \emph{modulus of identifiability}.
\end{definition}

\begin{remark}
    \label{rem:identifiable-slope}
	When $\M$ is a $\C^p$ manifold for some $p \geq 1$,
    since a $\C^p$ manifold $\mathcal{M}$ is locally closed,
	\cite[Prop.~3.12, Cor.~3.14]{lewis2024identifiability}
	indicate that
    \begin{equation*}
        \eta = \liminf_{x \rightarrow _f \bar x, x \notin \mathcal{M}} |\nabla f|(x).
    \end{equation*}
\end{remark}

It is shown in \cite[Prop.~10.12]{drusvyatskiy2012optimalityidentifiabilitysensitivity} and
\cite[Thm.~8.5]{lewis2024identifiability} that if $f$ is
$\C^2$-partly smooth at $\bar x$ for zero relative to $\mathcal{M}$
and the \emph{nondegeneracy condition}
\begin{equation}
	0 \in \operatorname{ri}\, \hat \partial f(\bar x)
	\label{eq:nondegeneracy}
\end{equation}
holds, then $\mathcal{M}$ is identifiable at $\bar x$.
We relax the condition from $\C^2$ to $\C^1$-partial smoothness in
\cref{lem:sharpness-under-C1-partly-smoothness} below.
Throughout the remainder of this section, we impose
the following assumption.
\begin{assumption}
    \label{ass:partial-smoothness}
    The function $f$ is proper, closed, and $\C^1$-partly
	smooth at $\bar x$ for $0$ relative to $\M \ni \bar x$, and the
	nondegeneracy condition
	\cref{eq:nondegeneracy} holds.
\end{assumption}
\begin{lemma}%
    \label{lem:sharpness-under-C1-partly-smoothness}
	For a function $f$ under \cref{ass:partial-smoothness},
	$\mathcal{M}$ is identifiable at $\bar x$ for $0$.
\end{lemma}

\begin{proof}
    Assume that the conclusion does not hold, then there exists a sequence $x_r \to_f \bar x$ with $x_r \notin \mathcal{M}$ such that $y_r \in \partial f(x_r) \to 0 \in \partial f(\bar x)$ as $r \to \infty$.
    Since $\mathcal{M}$ is $\C^1$, we can use
	\cref{lem:local-graph} to write $\mathcal{M}$ around $\bar x$ as a $\C^1$ graph over $T_{\bar x}\mathcal{M}$ with graph map $h:T_{\bar x}\mathcal{M}\cap B_{\epsilon}(0)\to N_{\bar x}\mathcal{M}$ satisfying $h(0)=0$ and $\nabla h(0)=0$.
    For each $x_r$, we have the decomposition $x_r-\bar x=p_r+q_r$ with $p_r \in T_{\bar x}\mathcal{M}$ and $q_r \in N_{\bar x}\mathcal{M}$,
	and we define
    \begin{equation*}
        x_r^{\prime} \coloneqq \bar x + p_r + h(p_r) \in \mathcal{M}.
    \end{equation*}
    Then $x_r^{\prime} \to \bar x$, $x_r \ne x_r^{\prime}$ for all
	large $r$, and for all such $r$ we have
    \begin{equation*}
        u_r \coloneqq
        \frac{x_r-x_r^{\prime}}{\|x_r-x_r^{\prime}\|}
        =
        \frac{q_r-h(p_r)}{\|q_r-h(p_r)\|}
        \in N_{\bar x}\mathcal{M}.
    \end{equation*}
    By taking a subsequence if necessary, we can assume that $u_r \to u \in N_{\bar x}\mathcal{M}$ with $\|u\| = 1$.
    Sharpness and nondegeneracy conditions imply that there exists
	$\delta > 0$ small enough such that $\delta u \in \hat \partial
	f(\bar x) \subseteq \partial f(\bar x)$ and $\delta u$ lies in the
	neighborhood of $0$ where the inner semicontinuity condition
	and \cref{eq:prox-reg-def} hold.
    Then, inner semicontinuity of $\partial f$ at $\bar x$ around $0$ on $\mathcal{M}$ implies that there exists $z_r \in \partial f(x_r^{\prime})$ such that $z_r \to \delta u$ as $r \to \infty$.
    By prox-regularity of $f$ at $\bar x$ for $0$, there exists
	$\rho\geq0$ such that for all sufficiently large $r$, we have
    \begin{equation*}
		\begin{cases}
			f(x_r) &\geq f(x_r^{\prime}) + \langle z_r, x_r -
			x_r^{\prime} \rangle - \frac{\rho}{2} \| x_r -
			x_r^{\prime} \|^2,\\
			f(x_r^{\prime}) &\geq f(x_r) + \langle y_r, x_r^{\prime} - x_r \rangle - \frac{\rho}{2} \| x_r - x_r^{\prime} \|^2.
		\end{cases}
    \end{equation*}
    Summing the two inequalities up gives
    \begin{equation*}
        0 \geq \langle z_r - y_r, x_r - x_r^{\prime} \rangle - \rho \| x_r - x_r^{\prime} \|^2.
    \end{equation*}
    Therefore, for all $r$ large enough, since $\| x_r - x_r^{\prime} \| > 0$, we have
    \begin{equation*}
        \rho \| x_r - x_r^{\prime} \| \geq \langle z_r - y_r, u_r \rangle.
    \end{equation*}
    Taking limits on both sides as $r \to \infty$ gives the following
	contradiction.
    \begin{equation*}
        0 \geq \langle \delta u - 0, u \rangle = \delta \|u\|^2 = \delta > 0,
    \end{equation*}
\end{proof}
By the lower semicontinuity of $f$,
\cref{lem:sharpness-under-C1-partly-smoothness} implies there
exist $\epsilon,\nu,\eta>0$ such that
\[
    \operatorname{dist}(0,\partial f(x))>\eta,
    \qquad
    \forall x\in
    B_{\epsilon}(\bar x)
    \cap [f\leq f(\bar x)+\nu]
    \setminus\mathcal M.
\]

We close this section by providing a linear-growth statement for identifiable
manifolds.
When an identifiable set $\M$ is a $\C^2$-smooth manifold, its
nearest-point projection is locally well defined and single valued.
\cite[Thm.
D.2]{davis2024stochastic} states that in this case, $f$ grows at least
linearly around $\bar x$ in the normal direction of $\mathcal{M}$,
i.e., there exist constants $\delta, \epsilon>0$ such that
\begin{equation*}
    f(x) \geq f(P_{\mathcal{M}}(x)) + \delta d(x, P_{\mathcal{M}}(x)), \quad \forall x \in B_{\epsilon}(\bar x).
\end{equation*}
Their result leverages the nearest-point projection and therefore
does not extend easily to the $\C^1$ setting, where the
nearest-point projection needs not be locally well defined.
Motivated by the sharp-minimizer result on $\C^2$ manifolds in
\cite[Thm.~6.1]{lewis2002active}, we develop in
\cref{thm:linear-growth} a linear-growth result
for $\C^1$ manifolds.

\begin{theorem}[Linear-growth under $\C^1$-partial smoothness]
    \label{thm:linear-growth}
    Consider $f:\mathbb{R}^n\to\overline{\R}$ satisfying \cref{ass:partial-smoothness}
	and set
    \begin{equation*}
        U\coloneqq T_{\bar x}\mathcal{M},
        \qquad
        V\coloneqq N_{\bar x}\mathcal{M},
    \end{equation*}
    then there are neighborhoods $U_0 \subseteq U$ and
	$V_0 \subseteq V$ of $0$ and a $\C^1$ map $v:U_0\to V$ such that $v(0)=0$, $\nabla v(0)=0$, and
    \begin{equation}
        \bar x+u+w\in\mathcal{M}
        \quad\Longleftrightarrow\quad
        w=v(u)
		\label{eq:graph-map}
    \end{equation}
    for all $u\in U_0$ and all $w\in V_0$. Moreover, there exists a constant $\kappa>0$ such that
    \begin{equation}
        \label{eq:uniform-normal-linear-growth}
        f(\bar x+u+w)
        \geq
        f(\bar x+u+v(u))+\kappa\|w-v(u)\|
    \end{equation}
    for all $u\in U_0$ and all $w\in V_0$.

    Consequently, every point $y$ sufficiently close to $\bar x$ can be written uniquely as $y=\bar x+u+w$ with $u\in U$ and $w\in V$, and there exists a point 
    \begin{equation*}
        R_{\bar x}(y)\coloneqq \bar x+u+v(u)\in\mathcal{M}
    \end{equation*}
    that satisfies
    \begin{equation*}
        f(y)\geq f(R_{\bar x}(y))+\kappa\|y-R_{\bar x}(y)\|
        \geq f(R_{\bar x}(y))+\kappa\operatorname{dist}(y,\mathcal{M}).
    \end{equation*}
\end{theorem}

\begin{proof}
	The part of \cref{eq:graph-map} follows directly from
	\cref{lem:local-graph}, and we thus focus on proving \cref{eq:uniform-normal-linear-growth}.
    If $V=\{0\}$, then $w=v(u)=0$ in the local representation and the
	conclusion is immediate. We therefore suppose $V\neq\{0\}$.
	Similarly, we also assume $w\neq v(u)$ below.

    Since $0\in\operatorname{ri} \hat \partial f(\bar x)$ and $\operatorname{par}\hat\partial f(\bar x)=V$, there exists $\delta>0$ such that
		\begin{equation}
			\label{eq:inclusion}
			B_\delta(0)\cap V\subseteq \hat \partial f(\bar x) \subseteq
			\partial f(\bar x).
		\end{equation}
    By prox-regularity of $f$ at $\bar x$ for $0$, there exist $\tau>0$ and $\rho\geq0$ such that
    \begin{equation}
        \label{eq:prox-regular-linear-growth}
        f(z)\geq f(y)+\langle s,z-y\rangle-\frac{\rho}{2}\|z-y\|^2
    \end{equation}
    whenever $y,z\in B_\tau(\bar x)$, $f(y)<f(\bar x)+\tau$, and $s\in\partial
	f(y) \cap B_{\tau}(0)$.
    For any $\alpha\in(0,\delta)$,
    we claim that there is a neighborhood $W$ of $\bar x$ such that
    \begin{equation}
        \label{eq:uniform-isc-linear-growth}
		\dist(\alpha d, \partial f(y)) \leq \frac{\alpha}{4},\quad
		\forall y \in W\cap\mathcal{M},\ \forall d\in \mathbb{S}^{n-1}\cap V.
	\end{equation}
    Indeed, if this fails, there must exist sequences $y_k\to \bar x$ in
	$\mathcal{M}$ and $d_k\in \mathbb{S}^{n-1} \cap V$ and $d \in
	\mathbb{S}^{n-1} \cap V$
	such that
	\begin{equation}
		\label{eq:uniform-isc-linear-growth-contradiction}
		\operatorname{dist}(\alpha d_k,\partial
		f(y_k))>\frac{\alpha}{4}, \quad \lim_{k \to \infty} \, d_k = d.
	\end{equation}
	Since $\alpha d\in\partial f(\bar x)$ by
	\cref{eq:inclusion} and that $\alpha < \delta$,
	the inner semicontinuity condition gives
	points $s_k\in\partial f(y_k)$ with $s_k\to\alpha d$, which
	contradicts \cref{eq:uniform-isc-linear-growth-contradiction} and
	thus \cref{eq:uniform-isc-linear-growth} holds.

	Now we select the neighborhood $U_0$ and $\epsilon > 0$ small
	enough so that:
	(a) $y\coloneqq \bar x+u+v(u) \in \M \cap W$ for the $W$ defined in the
	previous paragraph and satisfies $f(y)<f(\bar x)+\tau$ for all $u\in
	U_0$ (such $U_0$ exists by \cref{lem:local-graph}
	and the continuity of $f|_{\mathcal{M}}$),
	(b) $\bar x+u+w\in B_\tau(\bar x)$ and $y \in B_\tau(\bar x)$ whenever $u\in U_0$
	and $w\in B_\epsilon(0)\cap V$, and
	(c) $\|v(u)\|<\epsilon$ for all $u\in U_0$.
	When $\rho > 0$,
	we also require
	\begin{equation}
		\label{eq:epsilon-rho}
		2\epsilon\leq\alpha/(2\rho)
	\end{equation}
    We then fix $u\in U_0$ and $w\in B_\epsilon(0)\cap V$ with $w\neq v(u)$, set
    \[
        d=\frac{w-v(u)}{\|w-v(u)\|}\in \mathbb{S}^{n-1} \cap V, \quad
		y=\bar x+u+v(u)\in\mathcal{M},
    \]
	and choose $\alpha \in (0,\delta)$ small enough and $s \in
	\partial f(y)$ satisfying
	\cref{eq:uniform-isc-linear-growth}
	so that $\|s\| \leq \|s-\alpha d\| +\|\alpha d\|<2\alpha<\tau$.
	We can thus apply
	\cref{eq:prox-regular-linear-growth} with $z\coloneqq \bar x+u+w$ and $y$ and
	obtain
    \begin{align}
		\label{eq:linear-growth-1}
        f(\bar x+u+w)-f(\bar x+u+v(u))
        \geq
        \langle s,w-v(u)\rangle-\frac{\rho}{2}\|w-v(u)\|^2.
    \end{align}
	Since $\|d\|=1$,
    \begin{align*}
		\frac{\langle s,w-v(u)\rangle}{\norm{w - v(u)}}
        = \langle s,d\rangle
		\geq (\langle s-\alpha d,d\rangle+\alpha)
        \geq (\alpha-\|s-\alpha d\|)
		\geq \frac{3\alpha}{4}.
    \end{align*}
	When $\rho > 0$, our choice for $U_0$ and $\epsilon$ yields
	$\|w-v(u)\|\leq 2\epsilon$, so we get from
	\cref{eq:linear-growth-1}, the inequality above, and
	\cref{eq:epsilon-rho} that
    \[
        f(\bar x+u+w)-f(\bar x+u+v(u))
        \geq
        \frac{\alpha}{2}\|w-v(u)\|.
    \]
    Thus, \cref{eq:uniform-normal-linear-growth} holds with $\kappa=\alpha/2$.
	Finally, every point near $\bar x$ has a unique decomposition in the fixed orthogonal splitting $\mathbb{R}^n=U\oplus V$, and the last assertion follows because $y-R_{\bar x}(y)=w-v(u)$ and $\operatorname{dist}(y,\mathcal{M})\leq\|y-R_{\bar x}(y)\|$.
\end{proof}
Instead of the nearest-point projection, \cref{thm:linear-growth} uses
the map $R_{\bar x}$ constructed from a local graph representation of $\M$ in
\cref{lem:local-graph}, which needs not be a nearest-point projection.

A related abstract linear-growth result is given in
\cite[Thm.~4.5]{lewis2024identifiability} under normal embedding of
the identifiable set, continuity of the restricted function and its
slope, and the existence of an approximate projection.
Under \cref{ass:partial-smoothness}, these conditions can be verified using
our $R_{\bar x}$ constructed via the local graph view.
Their result, however, is localized in both distance and objective
value, whereas \cref{thm:linear-growth} gives a uniform estimate throughout a
sufficiently small spatial neighborhood under the direct geometric
assumptions of $\C^1$ partial smoothness and nondegeneracy.

\section{Transferability of Error Bound}
\label{sec:eb}

In this section, we present our main result, which addresses the
question posed in \cref{sec:intro} of whether an error bound in
the ambient space $\mathbb R^n$ is equivalent to the corresponding
error bound on the active manifold $\mathcal M$ under $\C^1$ partial
smoothness.
We investigate this question from two complementary perspectives.
Throughout this section, we assume that
$f:\mathbb{R}^n\to\overline{\R}$ is proper and closed.

Our first perspective is a geometric one that deals with a formulation
of the error bound based on subdifferential and Riemannian-gradient
residuals, which are commonly used to measure first-order stationarity
in nonsmooth and manifold optimization.
We first give the definition of such error bounds in the ambient space
and on the manifold.
Our definitions use a gauge function $\phi$ to describe the relation.
We call a function $\phi:\R_+ \to\R_+$ a \emph{gauge
function} if it is continuous, strictly increasing, and satisfies
$\phi(0)=0$ and $\lim_{t\to \infty}\phi(t)=\infty$.\footnote{The condition
	$\phi(t) \to \infty$ is not needed in this
	section, as error bounds are local properties.
	We impose it only to simplify the relation between error
	bounds and the KL property in \cref{sec:ebkl}.}
The local generalized error bounds on $(\mathbb{R}^n,d)$ and $(\mathcal{M},d)$
are defined as follows.

\begin{definition}[Generalized Error Bound (GEB) on $(\mathbb R^n,d)$]
\label{def:subdiff-eb}
Let $S\subseteq\mathbb R^n$ be a set and let $\bar x\in S$.
We say that $f$ satisfies the local (subdifferential) generalized error bound (GEB) property around
$\bar x$ (on $(\R^n, d)$) relative to $S$ if there exist $\epsilon,\nu,\theta>0$
and a gauge function $\phi$ such that
\begin{equation}
    \phi(\operatorname{dist}(x,S)) \leq \theta \operatorname{dist} (0, \hat \partial f(x)),\quad
\forall x\in B_\epsilon(\bar x)\cap[f\leq f(\bar x)+\nu].
\label{eq:eb}
\end{equation}
\end{definition}

\begin{definition}[Generalized Error Bound (GEB) on $(\mathcal M,d)$]
\label{def:subdiff-eb-m}
Let $S\subseteq\mathbb R^n$ be a set, $\M$ be a $\C^1$ manifold, and let $\bar x\in
S\cap\mathcal M$.
We say that $f$ satisfies the local (subdifferential) generalized error bound (GEB) property around
$\bar x$ on $(\mathcal M,d)$ relative to $S$ if there exist
$\epsilon,\nu,\theta > 0$ and a gauge function $\phi$ such that
\begin{equation}
    \phi(\operatorname{dist}(x,S\cap\mathcal M))
    \leq \theta \left\|\nabla_{\mathcal M} f(x)\right\|,
\quad \forall x \in (B_{\epsilon}(\bar x) \cap [f\le f(\bar x)+\nu]) \cap \mathcal{M} .
\end{equation}
\end{definition}

The complementary perspective considers a metric approach starting
from a formulation of the GEB based on the slopes
$|\nabla f|$ and $|\nabla(f|_{\mathcal M})|$ respectively for the
ambient space and the manifold $\mathcal M$.

\begin{definition}[Generalized Error Bound (GEB) on $(\mathbb R^n,d)$]
\label{def:slope-eb}
Let $S\subseteq\mathbb R^n$ be a set and let $\bar x\in S$.
We say that $f$ satisfies the local (slope) generalized error bound (GEB) property around
$\bar x$ relative to $S$ if there exist $\epsilon,\nu,\theta>0$
and a gauge function $\phi$ such that
\begin{equation}
    \phi(\operatorname{dist}(x,S)) \leq \theta |\nabla f|(x),\quad
\forall x\in B_\epsilon(\bar x)\cap[f\leq f(\bar x)+\nu].
	\label{eq:slope-eb}
\end{equation}
\end{definition}

\begin{definition}[Generalized Error Bound (GEB) on $(\mathcal M,d)$]
\label{def:slope-eb-m}
Let $S,\M\subseteq\mathbb R^n$ be two sets and let $\bar x\in
S\cap\mathcal M$.
We say that $f$ satisfies the local (slope) generalized error bound (GEB) property around
$\bar x$ on $(\mathcal M,d)$ relative to $S$ if there exist
$\epsilon,\nu,\theta > 0$ and a gauge function $\phi$ such that
\begin{equation}
    \phi(\operatorname{dist}(x,S\cap\mathcal M))
    \leq \theta |\nabla(f|_{\mathcal M})|(x),
\quad \forall x \in (B_{\epsilon}(\bar x) \cap [f\le f(\bar x)+\nu]) \cap \mathcal{M} .
	\label{eq:slope-eb-m}
\end{equation}
\end{definition}
More concretely, the set $S$ above is often taken as either $\crit f$ or
$\argmin f$ in the literature as well as in our analysis.

\begin{remark}
	\label{rem:sublevel}
	In the literature, error bounds and related conditions are often
	stated without our sublevel restriction $[f\le f(\bar x)+\nu]$.
	In other cases, they are usually assumed in a neighborhood of $S$,
	or in a region where a stationarity residual is small (as in
	\cite{luo1993error,tseng2009coordinate,li2018calculus,wen2017linear}).
	A main reason is that we need such restrictions to invoke
	prox-regularity and identifiability, both of which are stated only
	at points whose objective value is close to $f(\bar x)$.
	When $f$ is (weakly) convex, as is often assumed alongside such
	error bounds, these properties already hold at all nearby points
	and the restriction can be dropped.
\end{remark}

These two formulations are tightly connected.
In particular, on the manifold,
\cref{lem:slope-grad} gives $ |\nabla(f|_{\mathcal M})|(x)
=\|\nabla_{\mathcal M}f(x)\|$.
In the ambient space,
equivalence between \cref{eq:eb,eq:slope-eb}
for the standard linear error bound ($\phi =
\operatorname{Id}$)
is proven in \cite[Remark~3.6]{drusvyatskiy2021nonsmooth}; the same approximation
argument, together with the continuity of $\phi$, extends directly to our
generalized formulation after possibly shrinking the neighborhood
and the sublevel parameter.
We thus refer to both as the GEB in most situations, but
when we need to distinguish them, we will refer to the first one as
the \emph{subdifferential GEB} and the second one as the \emph{slope
GEB}.

Despite this connection, the two approaches reveal different
mechanisms behind the transfer of error bounds.  The geometric
approach identifies the common tangential projection of nearby
Fr\'echet subgradients with the Riemannian gradient on $\mathcal M$.
The metric approach instead relies on identifiability, an
off-manifold lower bound for the slope, and the linear-growth result
of \cref{thm:linear-growth}.  We develop the two approaches separately
to make these mechanisms explicit.
The four formulations and the relations we develop in this section are
summarized in \cref{fig:geb-roadmap}.

\begin{figure}[htbp]
\centering
\begin{tikzpicture}[
    node distance=1.15cm and 2.25cm,
    formulation/.style={align=center, text width=3.15cm, inner sep=4pt},
    relation/.style={-{Latex[length=2mm]}, semithick},
    equivalence/.style={<->, semithick},
    conditional/.style={densely dashed, semithick}
]
    \node[formulation] (ga)
        {$\operatorname{dist}(0,\hat\partial f)$\\[-1pt]
         {\scriptsize \cref{def:subdiff-eb}}};
    \node[formulation, right=of ga] (gm)
        {$\|\nabla_{\mathcal M}f\|$\\[-1pt]
         {\scriptsize \cref{def:subdiff-eb-m}}};
    \node[formulation, below=of ga] (sa)
        {$|\nabla f|$\\[-1pt]
         {\scriptsize \cref{def:slope-eb}}};
    \node[formulation, right=of sa] (sm)
        {$|\nabla(f|_{\mathcal M})|$\\[-1pt]
         {\scriptsize \cref{def:slope-eb-m}}};

    \node[font=\small\itshape]
        at ($(ga.north)+(0,3mm)$) {Ambient space};
    \node[font=\small\itshape]
        at ($(gm.north)+(0,3mm)$) {Active $\C^1$ manifold};
    \node[font=\small\itshape, align=center, text width=2.7cm]
        at ($(ga.west)+(-18mm,0)$) {Geometric residual};
    \node[font=\small\itshape, align=center, text width=2.7cm]
        at ($(sa.west)+(-18mm,0)$) {Metric residual};

    \draw[equivalence] (ga) -- node[above, font=\scriptsize]
        {\cref{thm:eb-equiv-geometric}} (gm);
    \draw[equivalence] (sa) -- node[below, font=\scriptsize]
        {\cref{thm:eb-equiv-metric}} (sm);
    \draw[equivalence] (sm) -- node[right, font=\scriptsize, align=left]
        {\cref{lem:slope-grad}} (gm);
    \draw[equivalence]
        ([xshift=-1mm]ga.south) --
        node[midway, left=1mm, font=\scriptsize, align=right]
        {\cite[Remark~3.6]{drusvyatskiy2021nonsmooth}\\
         (same argument)}
        ([xshift=-1mm]sa.north);
\end{tikzpicture}
\caption{Relation between the four formulations of generalized error
bounds.}
\label{fig:geb-roadmap}
\end{figure}

Our ambient error bounds are closely related to the
\emph{generalized metric subregularity} recently introduced by Li,
Mordukhovich, and
Zhu~\cite[Def.~3.1]{li2024generalizedsubregularity}, which
states
that for $S \subseteq \crit f$, there exist $\epsilon,\theta>0$ and an
admissible function $\hat\phi:\R_+\to \R_+$
(meaning $\hat\phi(0)=0$ and $\hat\phi(t)\to0$ only if $t\to0$)
such that
\begin{equation}
    \hat\phi(\dist(x,S))\leq\theta\,\dist(0,\partial f(x)),\quad
	\forall x\in B_\epsilon(\bar x).
	\label{eq:gms}
\end{equation}
From \cref{prop:slope}, we have $\dist(0,\partial f(x))\leq|\nabla
f|(x)$, and hence \cref{eq:gms} implies both \cref{eq:eb,eq:slope-eb} with
$\nu=\infty$, that is, without the sublevel restriction. %

In our analysis, we often make the following assumption of proper
separation of isocost surfaces, which is commonly assumed
in nonconvex settings for analyses involving the Luo--Tseng error
bound \cite{luo1993error} for the critical set in the literature
\cite{tseng2009coordinate,li2018calculus,wen2017linear}.
We will see in \cref{sec:discussion-application} that our generalized error
bound condition is equivalent to a generalization of the Luo--Tseng
error bound in the context of regularized optimization.
\begin{assumption}
    \label{ass:isocost}
We say that $f$ has the \emph{proper separation of isocost surfaces}
property on $S$ at $\bar x$ if there exists $\sigma>0$ such that
\begin{equation}
    \label{eq:isocost}
    f(y)=f(\bar x)
    \qquad
    \text{whenever }y\in S\text{ and }\|y-\bar x\|\leq\sigma.
\end{equation}
\end{assumption}
\cref{ass:isocost} holds automatically when $S = \argmin f$ and
also holds true for $S = \crit f$ of a convex function $f$.
This mild condition, or rather its conclusion that $f$ is constant
on $S$ near $\bar x$, is also reminiscent of the assumption needed for
the uniformized KL property in \cite[Lemma~6]{bolte2014proximal}.
A sufficient condition for \cref{ass:isocost} is that
$f$ has only finitely many critical values and is continuous on $\crit
f$ at $\bar x$.
The former condition is satisfied by the broad class of tame
functions (see \cite[Thm.~5.2]{ioffe2009invitation} and
\cite[Cor.~9]{bolte2007clarke}), which is the canonical example
on which the Kurdyka--\L ojasiewicz property holds.

The following lemma
will be useful in both perspectives.
\begin{lemma}
	Consider $f$ satisfying \cref{ass:partial-smoothness,ass:isocost}
	with $\bar x\in S\cap\mathcal M$ and $S \subseteq \crit f$.
    Then, for all $x \in \mathcal M$ near $\bar x$, 
    \begin{equation}
        \operatorname{dist}(x,S)=\operatorname{dist}(x,S\cap\mathcal M).
        \label{eq:geometric-dist-equal}
    \end{equation}
\end{lemma}

\begin{proof}
    The inequality ``$\leq$'' is immediate.
	If it is strict,
    there exists $y\in S \setminus \M$ with
        $d(x,y)<\operatorname{dist}(x,S\cap\mathcal M)
        \leq d(x,\bar x)$,
		and thus
    $d(y,\bar x)\leq d(y, x) + d(x,\bar x) \leq 2 d(x,\bar x)$.
	\cref{ass:isocost} then gives
    $f(y)=f(\bar x)$ for $x$ close enough to $\bar x$.
    However, \cref{lem:sharpness-under-C1-partly-smoothness} and the
	closedness of $f$ imply that there
    exists $\eta>0$ such that $\operatorname{dist}(0,\partial f(y))>\eta$
    for all $y \notin \mathcal{M}$ near $\bar x$ with $f(y)$ near $f(\bar x)$,
	contradicting with that $y \in \crit f$.
\end{proof}

\subsection{Main Result: \emph{Geometric} Perspective}
\label{sec:subdiff-eb}
We start from the formulation of error bounds defined by the Fr\'echet subdifferential and the
Riemannian gradient.
We adapt results from \cite{miller2005newton,daniilidis2006geometrical}
in
\cref{lem:u-gradient-property,lem:ri-subgradient,cor:u-gradient-projection}
below.
Although the definition of partial smoothness in
\cite{miller2005newton,daniilidis2006geometrical} is slightly
different from \cref{def:partial-smoothness},
their proofs are still valid with minor modifications.
For completeness, we provide proofs in \cref{proofs-vu}.
For a function $f$ $\C^1$-partly smooth at $\bar x$
for $0$ relative to $\mathcal M$,
we have
from \cref{def:partial-smoothness} and \cref{prop:sharpness-stability}
that
\[
    \operatorname{par}\hat\partial f(x)
    = N_x\mathcal M, \quad
    (N_x \M)^\perp
    = T_x\mathcal M
\]
for all $x\in\mathcal M$ near $\bar x$.
\begin{lemma}
    \label{lem:u-gradient-property}
If $f$ is proper, closed, and $\C^1$-partly smooth at $\bar x$
for $0$ relative to $\mathcal M$, then for any $x$ in $\mathcal M$
near $\bar x$,
    the projection of a Fr\'echet subgradient onto $T_x \M$ is independent of its choice. More precisely,
		$P_{T_x \M}(g)
        =P_{\operatorname{aff}\hat\partial f(x)}(0)$
	for every $g\in \hat \partial f(x)$.
\end{lemma}

Our next lemma extends the results in \cite[Lemma~2.11, Thm.~2.12]{miller2005newton} 
and \cite[Lemma~20]{daniilidis2006geometrical} to $\C^1$ partial smoothness.

\begin{lemma}
    \label{lem:ri-subgradient}
	If \cref{ass:partial-smoothness} holds for a function $f$,
	then $P_{T_x \M}(g) \in \operatorname{ri}\hat\partial f(x)$ for
	all $x\in\mathcal M$ near $\bar x$ and every
	$g\in\hat\partial f(x)$.
\end{lemma}

Combining \cref{lem:u-gradient-property,lem:ri-subgradient} with
\cref{eq:Rgrad-Fsubdiff} yields the following result.
\begin{lemma}
    \label{cor:u-gradient-projection}
	Under the setting of \cref{lem:ri-subgradient}, for all
    $x\in\mathcal M$ near $\bar x$, we have
    \begin{equation*}
        \nabla_{\mathcal M}f(x)=P_{\hat\partial f(x)}(0)
        \qquad\text{and}\qquad
        \operatorname{dist}(0,\hat\partial f(x))
        =\|\nabla_{\mathcal M}f(x)\|.
    \end{equation*}
\end{lemma}

We are now ready for our main result of the geometric proof of the
error bound transferability.
\begin{theorem}
    \label{thm:eb-equiv-geometric}
	Consider a function $f$ such that \cref{ass:partial-smoothness} holds. Suppose that
	$S \subseteq \crit f$, $\bar x\in S\cap\mathcal M$, and assume
	that \cref{ass:isocost} holds.
	Then the subdifferential GEB property around $\bar x$ relative to $S$ on
	$(\mathbb R^n,d)$ is equivalent to the
	subdifferential GEB property around $\bar x$ relative to $S$ on
	$(\mathcal M,d)$
	for the same $\phi$.
\end{theorem}

\begin{proof}
    For $x\in\mathcal M$ near $\bar x$, by
    \cref{cor:u-gradient-projection}, we have
    \begin{equation}
        \operatorname{dist}(0,\hat\partial f(x))
        =\|\nabla_{\mathcal M}f(x)\|.
		\label{eq:link}
    \end{equation}
    Moreover, by \cref{lem:sharpness-under-C1-partly-smoothness} and lower semicontinuity of $f$, there
    exist $\epsilon_0,\nu_0,\eta>0$ such that
    \begin{equation}
        \operatorname{dist}(0,\hat\partial f(x))
        \geq\operatorname{dist}(0,\partial f(x))>\eta, \quad
        \forall x \in (B_{\epsilon_0}(\bar x)
        \cap[f\leq f(\bar x)+\nu_0])\setminus\mathcal M.
        \label{eq:geometric-off-manifold-lowerbound}
    \end{equation}

    We now prove the two implications. Suppose first that the ambient
    GEB holds with parameters $(\epsilon,\nu,\theta,\phi)$. Shrink
    $\epsilon$ so that $\epsilon\leq\epsilon_0$ and
    $B_\epsilon(\bar x)\cap\mathcal M\subseteq
    [f\leq f(\bar x)+\nu]$. Then, for every
    $x\in B_\epsilon(\bar x)\cap\mathcal M$, 
    \[
        \phi(\operatorname{dist}(x,S\cap\mathcal M))
        = \phi(\operatorname{dist}(x,S))
        \leq\theta\operatorname{dist}(0,\hat\partial f(x))
        =\theta \|\nabla_{\mathcal M}f(x)\|,
    \]
	where the first equality follows from
	\cref{eq:geometric-dist-equal} and the last \cref{eq:link}.
    This is exactly the GEB on $(\mathcal M,d)$.

    Conversely, suppose now the GEB on $(\mathcal M,d)$ holds with
    parameters $(\epsilon,\nu,\theta,\phi)$. We shrink $\epsilon$
	small enough so that $\epsilon\leq\epsilon_0$ and (by the
	continuity of $\phi$)
    \begin{equation*}
        \phi(\|x-\bar x\|)\leq\theta\eta, \quad
        \forall x\in B_\epsilon(\bar x).
    \end{equation*}
    Set $\hat\nu\coloneqq\min\{\nu,\nu_0\}$ and consider any
    $x\in B_\epsilon(\bar x)\cap[f\leq f(\bar x)+\hat\nu]$.
    If $x\in\mathcal M$, then \cref{eq:geometric-dist-equal} and the
    manifold GEB give
    \[
        \phi(\operatorname{dist}(x,S))
        =\phi(\operatorname{dist}(x,S\cap\mathcal M))
        \leq\theta\|\nabla_{\mathcal M}f(x)\|
        =\theta\operatorname{dist}(0,\hat\partial f(x)).
    \]
    When $x\notin\mathcal M$, then since
    $\operatorname{dist}(x,S)\leq\|x-\bar x\|$, we have from
	\cref{eq:geometric-off-manifold-lowerbound} that
    \[
        \phi(\operatorname{dist}(x,S))
        \leq\phi(\|x-\bar x\|)
        \leq\theta\eta
        <\theta\operatorname{dist}(0,\hat\partial f(x)).
    \]
     Combining the two
    cases proves the ambient GEB with parameters
    $(\epsilon,\hat\nu,\theta,\phi)$.
\end{proof}

\subsection{Main Result: \emph{Metric} Perspective}
\label{sec:slope-eb}

We now turn to the metric formulation of error bounds formulated
via slope.
In \cref{thm:eb-equiv-metric} below, we
establish the equivalence between the slope GEB
property on $(\mathbb R^n,d)$ and its restriction to the active
manifold $(\mathcal M,d)$ under $\C^1$-partial smoothness and nondegeneracy
\cref{eq:nondegeneracy}.

\begin{theorem}
    \label{thm:eb-equiv-metric}
	Consider a function $f$ such that \cref{ass:partial-smoothness} holds. Suppose that	$S \subseteq \crit f$, $\bar x\in S\cap\mathcal M$, and assume
	that \cref{ass:isocost} holds.
	Then the slope GEB property around $\bar x$ relative to $S$ on
    $(\mathbb R^n,d)$ is equivalent to the slope GEB property around
    $\bar x$ relative to $S$ on $(\mathcal M,d)$
	for the same gauge function $\phi$.
\end{theorem}

\begin{proof}
By \cref{lem:sharpness-under-C1-partly-smoothness} and \cref{rem:identifiable-slope}, there exists $\eta
> 0$ such that 
\begin{equation}
    \label{eq:subdiff-slope-lowerbound}
    \operatorname{dist}(0, \partial f(x)) > \eta \text{ and } |\nabla f|(x) > \eta
\end{equation}
for all $x \notin \mathcal{M}$ near $\bar x$ with $f(x)$ near $f(\bar x)$.
We prove the two directions respectively.

\smallskip
\noindent\textbf{Case 1: GEB on $(\mathcal{M},d)$ implies GEB on
$(\mathbb{R}^n,d)$.}\\
We first assume that GEB on $(\M,d)$ holds at $\bar x$ with constants
$\epsilon,\nu,\theta >0$ and function $\phi$.
Fix any $x \in B_\epsilon(\bar x) \cap [f\le f(\bar x)+\nu]$.
If $x \in \mathcal{M}$, by the definition of GEB on $(\mathcal{M},d)$
and the definition of the respective slopes,
we have
\begin{equation*}
    \phi(\operatorname{dist}(x, S)) \leq \phi(\operatorname{dist}(x, S
	\cap \mathcal{M})) \leq \theta|\nabla (f|_{\mathcal{M}})|(x) \leq
	\theta
	|\nabla f|(x).
\end{equation*}
On the other hand, if $x \notin \mathcal{M}$,
\cref{eq:subdiff-slope-lowerbound} gives $|\nabla f|(x) > \eta$.
After shrinking $\epsilon$ if necessary, we have
$\phi(d(x, \bar x)) \leq \theta \eta$, and thus
\begin{equation*}
	\phi(\operatorname{dist}(x, S)) \leq
	\phi(d(x, \bar x)) \leq \theta \eta < \theta |\nabla f|(x).
\end{equation*}

\noindent\textbf{Case 2: GEB on $(\mathbb{R}^n,d)$ implies GEB on
$(\mathcal{M},d)$.}\\
Now suppose that GEB on $(\mathbb{R}^n,d)$ holds with constants
$\epsilon,\nu,\theta>0$ and function $\phi$.
By shrinking $\epsilon$ if necessary, we have
$\operatorname{dist}(x, S) = \operatorname{dist}(x, S \cap \mathcal{M})$
for all $x \in \mathcal{M} \cap B_{\epsilon}(\bar x)$ from
\cref{eq:geometric-dist-equal},
and
$B_\epsilon(\bar x)\cap\mathcal{M}\subseteq [f\le f(\bar x)+\nu]$ and
$\phi(\operatorname{dist}(x, S)) \leq \theta$ for all $x\in
B_\epsilon(\bar x)$
from that $f|_{\mathcal{M}}$ is continuous around $\bar x$.

If $\operatorname{dist}(x,S)=0$, the desired GEB inequality on
$(\mathcal{M},d)$ is trivially true by \cref{eq:geometric-dist-equal}.
Thus, we assume $\operatorname{dist}(x,S)>0$ onward.
By GEB on $(\mathbb{R}^n,d)$, $|\nabla f|(x)\geq
\theta^{-1} \phi(\operatorname{dist}(x,S))>0$.
Hence, the definition of slope indicates that there exists a sequence $x_r \to x$ with $x_r \ne x$ such that
\begin{equation}
\label{eq:func-decrease-x}
    \frac{1}{2} \phi(\operatorname{dist}(x, S)) \leq \theta \frac{f(x)-f(x_r)}{d(x,x_r)}
\;\Longleftrightarrow\;   \frac{1}{2} \phi(\operatorname{dist}(x, S))d(x,x_r) \leq \theta(f(x)-f(x_r)).
\end{equation}
From \cref{thm:linear-growth}, there is $\kappa > 0$ such that, 
for all sufficiently large $r$,
with $v_r\coloneqq R_{\bar x}(x_r)\in\mathcal{M}$
\begin{equation}
	\label{eq:linear-x}
    f(x_r) \geq f(v_r) + \kappa d(x_r, v_r).
\end{equation}
Substituting \cref{eq:linear-x} into \cref{eq:func-decrease-x} gives
\begin{equation*}
    \theta(f(x) - f(v_r)) \geq \frac{1}{2} \phi(\operatorname{dist}(x, S))d(x,x_r) + \theta \kappa d(x_r, v_r).
\end{equation*}
As $\phi(\dist(x,S)) \leq \theta$, we further obtain
\begin{equation*}
    \theta(f(x) - f(v_r))  \geq  \left( \frac{1}{2} d(x,x_r) + \kappa d(x_r,
	v_r)\right) \phi(\operatorname{dist}(x, S)).
\end{equation*}
Since $x_r \ne x$ and $\operatorname{dist}(x,S)>0$,
the inequality above implies that
 $f(v_r)<f(x)$, and hence $v_r \ne x$ for all $r$, so division by
 $d(x,v_r)$ is always valid.
Let $\gamma \coloneqq \min\{1/2,\kappa\}$, then the triangle inequality gives
\begin{equation*}
	\theta \frac{f(x) - f(v_r)}{d(x,v_r)} \geq \gamma
	\phi(\operatorname{dist}(x, S)).
\end{equation*}
The local graph representation in \cref{lem:local-graph} shows that
near $\bar x$, $R_{\bar x}$ is continuous and $R_{\bar x} =
\operatorname{Id}$ on $\mathcal{M}$, so
\begin{equation*}
    v_r=R_{\bar x}(x_r)\to R_{\bar x}(x)=x.
\end{equation*}
Therefore, we get
\begin{equation*}
	\phi(\operatorname{dist}(x, S)) \leq \gamma^{-1}\theta |\nabla (f|_{\mathcal{M}})|(x).
\end{equation*}
Since $\operatorname{dist}(x, S) = \operatorname{dist}(x, S \cap \mathcal{M})$, we
conclude that GEB indeed holds on $(\mathcal{M},d)$ around $\bar x$ with
$\gamma^{-1} \theta$.
\end{proof}
\begin{remark}
The proof of \cref{thm:eb-equiv-metric} in the manifold-to-ambient
direction does not utilize any property beyond identifiability.
It is hence applicable to any identifiable set in a complete metric
space.
On the contrary, the geometric approach in
\cref{thm:eb-equiv-geometric} uses Riemannian gradient and thus does
not extend beyond $\C^1$ active manifolds.
\end{remark}

When $S = \argmin f$, \cref{ass:isocost} holds automatically, and we
get the following corollary.
\begin{corollary}
\label{cor:argmin-eb-equiv}
Consider the case that $S = \argmin f$ and let $\bar x\in S\cap\mathcal M$.
Suppose that \cref{ass:partial-smoothness} holds.
Then the GEB property around $\bar x$ relative to $S$ on
$(\mathbb R^n,d)$ is equivalent to the GEB property around
$\bar x$ relative to $S$ on $(\mathcal M,d)$, i.e.,
\cref{def:subdiff-eb} $\Leftrightarrow$ \cref{def:subdiff-eb-m} and
\cref{def:slope-eb} $\Leftrightarrow$ \cref{def:slope-eb-m}.
\end{corollary}

We close this section by remarking on the usefulness of our
ambient-to-manifold direction in
\cref{thm:eb-equiv-geometric,thm:eb-equiv-metric,cor:argmin-eb-equiv}.
For the KL property, known applications use only the
manifold-to-ambient direction,
because the usual sufficient conditions, such as tameness, for
ambient KL are also sufficient for manifold KL (although possibly with
a different desingularizer).
For error bounds, by contrast, no such generic mechanism is available,
and our results turn known ambient error bounds (e.g.,
\cite{luo1993error,tseng2010approximation,hou2013linear,zhou2017unified}),
into sufficient conditions for the on-manifold error bounds required
by methods like those in \cite{Lee2023a,hu2025analysis} for fast local
convergence.

\section{The KL Property}
\label{sec:kl}

We now turn to the transferability of the Kurdyka--\L ojasiewicz (KL) property.
We first recall the slope formulation of the KL property from
\cite{bolte2010characterizations,lewis2024identifiability}.
A function $\psi:\R_+ \to \R_+ $ is called a \emph{desingularizer} if
it is continuous on $\R_+$, is continuously differentiable with
$\psi'(t)>0$ on $(0,\infty)$, and satisfies $\psi(0)=0$ and $\lim_{t
\uparrow \infty} \psi(t) = \infty$.

\begin{definition}[KL property on a metric space
		{\cite[Def.~5.8]{lewis2024identifiability}}]
\label{def:kl}
	Let $(Y,\bar d)$ be a metric space.
We say that a closed function $g:Y\to\overline{\R}$
has the \emph{KL property at $\bar y\in\operatorname{dom}g$ with desingularizer $\psi$} if there exist constants $\epsilon,\delta>0$ such that
\begin{equation}
	\label{eq:kl-slope-definition}
	|\nabla (\psi(g(\cdot)-g(\bar y)))|(y)\geq \delta,\quad \forall y \in B_{\epsilon}(\bar y) \cap [g(\bar y)< g <g(\bar y)+\epsilon],
\end{equation}
where both the slope and the ball radius are computed using the metric $\bar d$ on $Y$.
\end{definition}
As noted in \cite[Sec.~5]{lewis2024identifiability},
we also have the following equivalent form for
\cref{eq:kl-slope-definition}.
\begin{equation}
	\label{eq:kl-derivative-definition}
	\psi^{\prime}(g(y)-g(\bar y))\,|\nabla g|(y) \geq \delta,\quad \forall y \in B_{\epsilon}(\bar y) \cap [g(\bar y)< g <g(\bar y)+\epsilon].
\end{equation}

In \cite[Thm.~3.7]{li2018calculus}, Li and Pong showed that in
Euclidean spaces, if the KL property holds on the active manifold at a
point for which $f$ is partly smooth, it also holds in the ambient
space with the same exponent $\alpha\in[0,1)$.
Lewis and Tian further studied in
\cite[Thm.~5.12]{lewis2024identifiability} the equivalence between
the KL property in a complete metric space and the KL property with
the same desingularizer on an identifiable set
(not necessarily a manifold)
using a linear-growth property established in
\cite[Thm.~4.2]{lewis2024identifiability}.
Their linear-growth property states that for a closed function $f$ on
a complete metric space $(Y, \bar d)$, a point $\bar x$ with $|\nabla
f|(\bar x) = 0$, and an identifiable set $\M\ni \bar x$, for any
$\epsilon > 0$ smaller than the modulus of identifiability and any
sequence $x_r \to_f \bar x$, there is a sequence $v_r \to_f \bar x$ in
$\M$ such that
\begin{equation}
	f(x_r) \geq f(v_r) + \epsilon \bar d(x_r,v_r),\quad \forall r \text{ large}.
	\label{eq:linear-lewis}
\end{equation}
Compared with \cref{thm:linear-growth},
their estimate has more flexibility in the linear-growth
parameter and applies in a more general space to a more general
identifiable set.
It is, however, pointwise:
it holds only at the reference point $\bar x$ rather than uniformly at
nearby points of the identifiable set.
Such a pointwise estimate is insufficient for transferring the ambient
slope inequality to $f|_{\mathcal M}$ at an arbitrary nearby point
$x\in\mathcal M$, and we therefore develop
\cref{thm:linear-growth}.
In particular, in \cref{eq:func-decrease-x} and \cref{eq:linear-x}, we
considered a point $x$ near $\bar x$ but with a positive slope, and
thus \cref{eq:linear-lewis} does not apply.
The proof of \cite[Thm.~5.12]{lewis2024identifiability} encounters precisely
the same issue in the ambient-to-restricted direction.
Our counterexample in \cref{sec:counterexample} shows this
implication does not hold in such generality
(while the manifold-to-ambient direction remains valid),
and additional regularity of $\mathcal M$ and $f|_{\mathcal M}$ is
needed.

\subsection{Ambient-to-Manifold Transfer for KL}
\label{subsec: KL}
We first show that $\C^1$ partial smoothness provides the
additional regularity needed for the ambient-to-manifold direction of
the KL property.
Our proof mainly follows the argument of Lewis and
Tian~\cite[Thm.~5.12]{lewis2024identifiability},
with their pointwise linear-growth result \cite[Thm.~4.2]{lewis2024identifiability} replaced by our uniform linear-growth result in
\cref{thm:linear-growth}.
For completeness and for contrasting the two directions, we also include
in the theorem statement the manifold-to-ambient part from \cite[Thm.~5.12]{lewis2024identifiability}.

\begin{theorem}%
    \label{thm:kl-ambient-to-manifold}
	Let $\psi$ be a desingularizer locally concave around $0$ (or,
	more generally, satisfying $\liminf_{t\downarrow0}\psi'(t)>0$).

	\begin{enumerate}
		\item[(a)]
		Let $f:\mathbb{R}^n\to\overline{\R}$ satisfy
		\cref{ass:partial-smoothness}.
		If $f$ satisfies the KL property at $\bar x$ with
		desingularizer $\psi$ on $(\mathbb{R}^n,d)$,
		then $f|_{\mathcal M}$ satisfies the KL property at $\bar x$ with the same desingularizer on $(\mathcal M,d)$.

	\item[(b)]
	On a complete metric space $(Y,\bar d)$, consider a closed function $f$ that has slope zero at a point $\bar x$ and let $\M$ be an identifiable set at $\bar x$.\footnote{The definition of identifiability in a complete metric space $(Y, \bar d)$ replaces the distance $\operatorname{dist}(0, \partial f(x))$ with the slope $|\nabla f|(x)$ and the Euclidean distance with $\bar d$.}
	If $f|_{\M}$ has the KL property at $\bar x$ with desingularizer $\psi$ on $(\M,\bar d)$, then $f$ also has the KL property at $\bar x$ with the same desingularizer on $(Y,\bar d)$.
\end{enumerate}
\end{theorem}

\begin{proof}
	We only prove (a), amd (b) follows from
	\cite[Thm.~5.12]{lewis2024identifiability}.
    Without loss of generality, we may assume $f(\bar x)=0$.
	Thus \cref{eq:kl-slope-definition} becomes
    \begin{equation*}
        |\nabla (\psi \circ f)|(x)\geq \delta,
	\quad \forall x \in B_{\epsilon}(\bar x) \cap [ 0 < f < \epsilon].
    \end{equation*}
	After shrinking $\epsilon$ and $\delta$ if necessary,
	the assumption on $\psi'$ implies
	\begin{equation}
		\label{eq:psilower}
		\psi^{\prime} (t)> \delta, \quad \forall t \in (0,\epsilon).
	\end{equation}
    \Cref{thm:linear-growth} gives a constant $\kappa>0$, a neighborhood $W$ of $\bar x$, and a continuous map $R_{\bar x}:W \to\mathcal M$ such that
    \begin{equation}
        \label{eq:linear-growth-for-kl}
        f(z)\geq f\bigl(R_{\bar x}(z)\bigr)
        +\kappa\|z-R_{\bar x}(z)\|,\quad \forall z \in W.
    \end{equation}
    Moreover, $R_{\bar x}(z)=z$ for all $z\in W\cap\mathcal M$.
	Therefore, after shrinking $\epsilon$ if necessary so that $B_{\epsilon}(\bar x)\subset W$,
	we have $R_{\bar x}(x)=x$
	for any $x \in \mathcal{M} \cap B_{\epsilon}(\bar x) \cap [0<f(x)<\epsilon]$.

	Now we pick an arbitrary $x \in \mathcal{M} \cap B_{\epsilon}(\bar x) \cap [0<f(x)<\epsilon]$.
	By the definition of slope, there exists a sequence $x_r \to x$ with $x_r \ne x$ such that
    \begin{equation}
        \psi(f(x))-\psi(f(x_r)) > \frac{1}{2} \delta d(x,x_r) >
		0,\quad \forall r \geq 0.
		\label{eq:half-slope}
    \end{equation}
    We first show that $0< f(x_r)<\epsilon$ for all sufficiently large $r$.
    Since $\psi$ is strictly increasing, we have $f(x_r)<f(x)<\epsilon$ for all $r$.
	By \cref{eq:linear-growth-for-kl}, let $v_r = R_{\bar x}(x_r)$, then $v_r \in \mathcal{M}$ and $v_r \to R_{\bar x}(x) = x$.
    Since $f|_{\mathcal M}$ is continuous by partial smoothness, we have $f(v_r)\to f(x)$.
	Hence, $0<f(v_r)<\epsilon$ for all sufficiently large $r$.
    By \cref{eq:linear-growth-for-kl}, we have $0<f(v_r)\leq f(x_r)<\epsilon$ for all sufficiently large $r$.
	By \cref{eq:half-slope} and the mean value theorem, there exists
	$\xi_r\in [f(v_r), f(x_r)]$ such that
    \begin{align*}
        \psi(f(x))-\psi(f(v_r)) & > \frac{1}{2} \delta d(x,x_r) + \psi(f(x_r)) - \psi(f(v_r))\\
        & = \frac{1}{2} \delta d(x,x_r) + \psi^{\prime}(\xi_r) \bigl(f(x_r)-f(v_r)\bigr).
    \end{align*}
	As $\xi_r \in [f(v_r),f(x_r)] \subset (0,\epsilon)$ for all
	sufficiently large $r$, \cref{eq:psilower} holds.
	By inserting \cref{eq:psilower,eq:linear-growth-for-kl} into the
	inequality above and defining
    $\gamma \coloneqq \delta\min\{1/2, \kappa\}$, we obtain
    \begin{align}
        \psi(f(x))-\psi(f(v_r))
        &> \frac{1}{2}\delta d(x,x_r)
        + \psi^{\prime}(\xi_r)\kappa d(x_r,v_r) \nonumber\\
        &\geq \frac{1}{2}\delta d(x,x_r)
        + \delta\kappa d(x_r,v_r) \nonumber\\
        &\geq \gamma d(x,v_r).
		\label{eq:philowerbound}
	\end{align}
	From the first inequality in \cref{eq:philowerbound} and that $x_r
	\ne x$, we get $\psi(f(x))-\psi(f(v_r)) > 0$, and thus $v_r
	\ne x$, for all sufficiently large $r$.
    Combining this with $v_r \in \mathcal{M}$ and $v_r \to x$, we conclude that
    \begin{equation*}
        |\nabla (\psi \circ f|_{\mathcal{M}})|(x) \geq \gamma.
    \end{equation*}
\end{proof}

\begin{remark}
\label{rem:metric-transfer}
As discussed after the proof of \cref{thm:linear-growth},
a related linear-growth estimate is given in
\cite[Thm.~4.5]{lewis2024identifiability} for points near the reference
point $\bar x$ in terms of both distance and function value.
With an additional argument using lower semicontinuity of $f$, nearby
points with objective value away from the reference value can be
handled as well.
Using their Thm.~4.5, we can see that the ambient-to-identifiable-set
direction of \cite[Thm.~5.12]{lewis2024identifiability} holds in a
complete metric space when, around the reference point, the
identifiable set $\M$ is normally embedded (not necessarily a
manifold), $f|_{\M}$ is continuous, $|\nabla (f|_{\M})|$ is continuous
at $\bar x$, and an approximate projection onto the
identifiable set $\M$ exists.
The last condition can be guaranteed by local closedness of $\M$.

In the same spirit,
our slope error bound analysis in \cref{sec:slope-eb} can also be
extended to the same setting of complete metric spaces and
identifiable sets beyond manifolds by
\cite[Thm.~4.5]{lewis2024identifiability} together with an additional
argument using the lower semicontinuity of $f$.
\end{remark}

\subsection{GEB Implies KL}
\label{sec:ebkl}
Finally, we show that the GEB property
for $\crit f$ together with the separation of isocost surfaces
property and prox-regularity implies the KL property.
This result does not rely on identifiability or partial smoothness,
and is thus of independent interest.
Before stating the proposition,
we need to define prox-regularity for the restriction
of a function to a manifold on which it is differentiable.
We say $f|_{\M}$ is prox-regular at $\bar x$ if there are $\epsilon'
> 0$ and $\rho \geq 0$ such that
\begin{equation}
		f|_{\M}(y) \geq f|_{\M}(x)
		+ \langle \nabla_{\M} f(x), y-x\rangle
		- \frac{\rho}{2}\|y-x\|^2,
		\, \forall x,y \in B_{\epsilon'}(\bar x) \cap \M.
	\label{eq:prox-reg-m}
\end{equation}
\begin{proposition}
	Consider a closed function $f: \R^n \to \overline{\R}$ and a
	point $\bar x \in S \subseteq \crit
	f$.
	\begin{enumerate}
		\item[(a)] Assume there are $\epsilon, \nu, \theta > 0$ and a gauge
			function $\phi$
			such that the GEB property
			\cref{eq:slope-eb} holds around $\bar x$ on $(\R^n,d)$,
			\cref{ass:isocost} holds,
			and $f$ is prox-regular at $\bar x$ for $0$ with constant
			$\rho \geq 0$.
			Let
			\begin{equation}
					\Phi(t) \coloneqq t\,\phi^{-1}(\theta t) +
					\frac{\rho}{2}\phi^{-1}(\theta t)^2,\quad \psi(t)
					\coloneqq \int_0^t \frac{ds}{\Phi^{-1}(s)}.
				\label{eq:psi-def}
			\end{equation}
If the improper integral above is finite for some $t>0$, then
$f$ has the KL property around $\bar x$ with desingularizer $\psi$.
		\item[(b)] Consider a $\C^1$-embedded submanifold $\M \subseteq
			\R^n$ around $\bar x \in \M$ such that $f|_{\M}$ is
			$\C^1$-smooth around $\bar x$.
			Assume there are $\epsilon, \nu, \theta > 0$ and a gauge
			function $\phi$ such that the GEB property
			\cref{eq:slope-eb-m} holds around $\bar x$ on $(\M,d)$,
			\cref{ass:isocost} holds,
			and $f|_{\M}$ is prox-regular at $\bar x$ with constant
			$\rho \geq 0$.
			Define $\Phi$ and $\psi$ by \cref{eq:psi-def} with the
			present $\phi,\theta,\rho$.
If the improper integral defining $\psi$ is finite for some $t>0$,
then $f|_{\M}$ has the KL property around $\bar x$ on $(\M,d)$ with
desingularizer $\psi$.
	\end{enumerate}
	\label{prop:eb-kl}
\end{proposition}
	Since $\phi$ is continuous, strictly increasing, and strictly
	positive on $(0,\infty)$, it is straightforward to verify that
	$\Phi$ is well defined and satisfies these properties of $\phi$.
\begin{proof}
	\textbf{(a)}.
We first decrease $\epsilon$, if necessary, so that
$\epsilon \leq \nu$ in \cref{eq:slope-eb},
$2\epsilon \leq \sigma$ for the $\sigma$ in
\cref{eq:isocost}, and \cref{eq:prox-reg-def} holds on
$B_{2\epsilon}(\bar x)$ with subgradients in $B_{2\epsilon}(0)$ with
constant $\rho$.
We now fix this $\epsilon$ and choose
\[
    \eta \coloneqq
    \min\{\epsilon,\Phi(\epsilon)\}/2 >0.
\]
We will prove the KL property on
    $B_\eta(\bar x)
    \cap [f(\bar x)<f<f(\bar x)+\eta]$,
while retaining $\epsilon$ as the threshold separating
the two slope cases.

Fix $x$ in this set.
If $|\nabla f|(x)=\infty$, the KL property is immediate,
so we assume $|\nabla f|(x)<\infty$.
Suppose first that $|\nabla f|(x)<\epsilon$.
Since $\partial f(x)$ is nonempty and closed, we can choose
$v \in \argmin_{s\in\partial f(x)}\|s\|$.
Then by \cref{prop:slope}, we have $\|v\|\leq|\nabla f|(x)<\epsilon$, so this subgradient
lies in the range required by prox-regularity.

Consider a sequence $y_r\in S$ such that
$d(x,y_r)\to d(x,S)$.
Since $\bar x\in S$,
\[
    d(x,S)\leq d(x,\bar x)<\eta\leq\epsilon/2.
\]
After discarding finitely many terms, we may assume
$d(x,y_r)<\epsilon$ for all $r$.
Consequently,
\[
    d(\bar x,y_r)
    \leq d(\bar x,x)+d(x,y_r)
    <\eta+\epsilon<2\epsilon\leq\sigma,
\]
and thus, by \cref{eq:isocost},
\[
    f(y_r)=f(\bar x)\qquad\text{for all }r.
\]
Moreover,
\[
    x,y_r\in B_{2\epsilon}(\bar x)
    \qquad\text{and}\qquad
    f(x)<f(\bar x)+\eta<f(\bar x)+\epsilon.
\]
Therefore, \cref{eq:prox-reg-def} gives
\begin{align}
    f(x)-f(\bar x)
    &=f(x)-f(y_r) \nonumber\\
    &\leq-\langle v,y_r-x\rangle
        +\frac{\rho}{2}d(x,y_r)^2 \nonumber\\
    &\leq\|v\|d(x,y_r)
        +\frac{\rho}{2}d(x,y_r)^2.
    \label{eq:prox-reg0}
\end{align}
Using \cref{prop:slope} again and taking the limit as $r\to\infty$, we
obtain
\begin{equation}
    f(x)-f(\bar x)
    \leq |\nabla f|(x)d(x,S)
        +\frac{\rho}{2}d(x,S)^2.
    \label{eq:prox-reg}
\end{equation}
Since $\eta\leq\epsilon\leq\nu$, the error bound
\cref{eq:slope-eb} applies at $x$ and yields
\[
    d(x,S)\leq\phi^{-1}(\theta|\nabla f|(x)).
\]
Substituting this into \cref{eq:prox-reg} gives
\begin{equation}
        f(x)-f(\bar x)
        \leq |\nabla f|(x)\,
            \phi^{-1}(\theta|\nabla f|(x))
            +\frac{\rho}{2}
            \phi^{-1}(\theta|\nabla f|(x))^2
        =\Phi(|\nabla f|(x)).
    \label{eq:KL_bound1}
\end{equation}

Next, suppose that $|\nabla f|(x)\geq\epsilon$.
Our choice of $\eta$ gives
\begin{equation}
    f(x)-f(\bar x)
    <\eta
    <\Phi(\epsilon)
    \leq\Phi(|\nabla f|(x)).
    \label{eq:KL_bound2}
\end{equation}

Combining \cref{eq:KL_bound1,eq:KL_bound2} and applying
the increasing map $\Phi^{-1}$, we conclude that
\[
    \Phi^{-1}(f(x)-f(\bar x))
    \leq|\nabla f|(x),\quad \forall
    x\in B_\eta(\bar x)
    \cap[f(\bar x)<f<f(\bar x)+\eta].
\]
Therefore,
\[
    \psi^\prime(f(x)-f(\bar x))\,|\nabla f|(x)
    =\frac{|\nabla f|(x)}
    {\Phi^{-1}(f(x)-f(\bar x))}
    \geq1.
\]
This proves \cref{eq:kl-derivative-definition} with neighborhood
parameter $\eta$ and $\delta=1$.

It remains to verify that $\psi$ is a desingularizer.
The assumptions on $\phi$ clearly imply that
$\Phi^{-1}$ is well defined, continuous and strictly increasing
on $\R_+$, and is positive on $(0,\infty)$.
Finiteness of the improper integral defining $\psi$ at
some positive endpoint implies its finiteness at every
finite endpoint and $\lim_{t\downarrow0}\psi(t)=0$.
Consequently, $\psi$ is continuous at $0$ with $\psi(0)=0$,
and $\psi'(t)=1/\Phi^{-1}(t)>0$ for all $t>0$.
Finally, since $\phi^{-1}(\theta t)\to\infty$,
there exists $T>0$ such that
$\phi^{-1}(\theta t)\geq1$ and thus
$\Phi(t)\geq t$ for $t\geq T$, or equivalently $\Phi^{-1}(s)\leq s$
for $s\geq\Phi(T)$.
It follows that, for $t\geq\Phi(T)$,
\[
    \psi(t)
    \geq \psi(\Phi(T))
        +\int_{\Phi(T)}^t\frac{ds}{s}
		\stackrel{t \to \infty}{\longrightarrow} \infty.
\]
Thus $\psi$ satisfies all the requirements of a
desingularizer.

		\textbf{(b)}.
		The proof is similar to part~(\textbf{a}), and since the Riemannian
		gradient is continuous, we do not need to consider the case
		that the slope is larger than the prox-regularity threshold.
\end{proof}
The condition \cref{eq:prox-reg-m} might seem unnatural at a
glance.
However, \cref{cor:u-gradient-projection} shows that it holds
automatically under \cref{ass:partial-smoothness}
(prox-regularity is implied by partial smoothness).

\cref{prop:eb-kl} recovers the widely seen power-type KL property
with desingularizers $\psi(t) = c t^{1-\alpha}$.
For a gauge $\phi(t)=t^q$,
$\rho > 0$ and $q \in (0,2)$ yield the KL exponent
$\alpha=\max\{q/(q+1),q/2\}$, which ranges over
$(0,1)$, while $\rho =0$ gives $\alpha = q/(q+1)$ for all $q >
0$.
A related result in \cite[Thm.~4.1]{li2018calculus} shows
that in regularized optimization, the Luo--Tseng error bound
implies the KL exponent $1/2$.
By using regular-variation techniques for inverse-integral
constructions in \cite{liu2026concrete}, we can also obtain explicit
asymptotic forms of desingularizers beyond the KL exponent type.
For example, when $\rho>0$ and
$\phi(t)=t^2(\log(1/t))^\beta$ near $0$ with $\beta>1$,
the construction in \cref{eq:psi-def} yields a desingularizer of order
$(\log(1/t))^{1-\beta}$ as $t\downarrow0$.

Due to the term involving $\phi^{-1}(\theta t)^2$ (arising from
prox-regularity) in $\Phi$, the conversion from
$\phi$ to $\psi$ in \cref{prop:eb-kl} may lose some sharpness,
and we can often get faster convergence rate guarantees from analyzing
through the GEB property than through the KL property.
This is consistent with the observation in
\cite[Sec.~5]{li2024generalizedsubregularity}.

Given the result of \cref{prop:eb-kl}, one might wonder if
the KL property also implies the generalized error bound property.
For this direction, \cite[Thm.~3.7]{drusvyatskiy2021nonsmooth}
shows that if the KL property (with $\psi(t) = c t^{1-\alpha}$ for $\alpha \in
(0,1)$ and $c > 0$) holds on $U \cap [f^* < f < r]$ for some $f^* < r$ and some
open set $U$, then it implies a generalized error bound condition in
the form of \cref{eq:slope-eb},
but relative to the set $S = [f \leq f^*]$, which may not be a
subset of $\crit f$.
Rebjock and Boumal \cite{rebjock2025fast} established
equivalence for $\C^2$ functions near local minima, with KL exponent
$1/2$.
Except for these special cases, the reverse implication from KL to GEB
is generally not true.
The reason is clear: the KL property controls the behavior of
only $x$ with $f(x) > f(\bar x)$, while the (generalized) error bound
also includes those $x$ with $f(x) < f(\bar x)$ beyond the region
governed by the KL property.

\section{Discussion and Application}
\label{sec:discussion-application}
We now illustrate the application of our main results on error bounds
with the two motivating examples discussed in \cref{sec:intro}.
We first treat the $\ell_p$-norm ball constrained problem in
\cref{eq:intro-lp-example} whose active manifold is not $\C^2$,
and then
show that for the $\ell_1$-regularized case, our result
recovers the proximal error bound equivalence around $\argmin f$ in
\cite{wu2025resolution} as a special
case.\footnote{\cite{wu2025resolution} proved the ambient-to-manifold
direction, while the manifold-to-ambient direction is straightforward
as we can see in the proof of \cref{thm:eb-equiv-metric}.}

\subsection{An \texorpdfstring{$\ell_p$}{lp}-ball Constrained Example}
\label{sec:lp-example}
Consider \cref{eq:intro-lp-example} with a differentiable $g$, and let
$\bar x$ be the nondegenerate critical point of interest.
Whenever $\bar x$ is sparse, the active manifold $\M = \{x :
\norm{x}_p = r\}$ is only $\C^{1,p-1}$ near $\bar x$ and not $\C^2$,
so our \cref{thm:linear-growth} is indeed needed, and
\cref{thm:eb-equiv-geometric,thm:eb-equiv-metric} provide the GEB
equivalence in $\R^n$ and on $\M$.
We consider the following simple example for demonstration.

Take $n = 2$, $p \in (1,2)$, and $g(x) = -x_1$ in
\cref{eq:intro-lp-example},
and
write $C \coloneqq \{x : \norm{x}_p \le
r\}$ for its feasible set.
We can then equivalently write this problem as $\min f(x)$ with $f = g
+ \delta_{C}$, where $\delta_C$ is the indicator function of $C$.
Clearly, for this convex problem, $S = \crit f = \argmin f = \{\bar x\}$
with $\bar x \coloneqq (r,0)$.
Moreover, \cref{ass:partial-smoothness} holds at $\bar x$.
Indeed, $f$ is convex and thus prox-regular and subdifferentially
regular everywhere, and the restriction $f|_{\M} = g|_{\M}$ is $\C^1$.
Writing $F(x) \coloneqq |x_1|^p + |x_2|^p - r^p$, we have
\[
    \partial f(x) = \nabla g(x) + \R_+ \nabla F(x)
\]
for $x \in \M$.
As $\nabla g$ and $\nabla F$ are continuous, this identity gives both
inner semicontinuity along $\M$ and the sharpness condition
\[
    \operatorname{par} \hat\partial f(\bar x)
    = \operatorname{span} \nabla F(\bar x) = N_{\bar x}\M.
\]
Finally, $\nabla g(\bar x) = -e_1$ and $\nabla F(\bar x) = p r^{p-1}
e_1$ give $\hat\partial f(\bar x) = \{(t-1)e_1 : t \ge 0\}$, whose
relative interior contains $0$.

We now compute both sides of the GEB on $(\M,d)$ for $x \in \M$ near
$\bar x$, writing $x_1 = (r^p - |x_2|^p)^{1/p} > 0$.
By \cref{prop:subdiff-inclusion},
\[
    \nabla_\M f(x) = P_{T_x\M}(\nabla g(x)),
\]
and $T_x \M = (\nabla F(x))^\perp$ with $\nabla F(x) = p(a,b)$
for $a \coloneqq x_1^{p-1}$ and $b \coloneqq \sign{x_2}|x_2|^{p-1}$.
Since $\nabla g(x) = -e_1$, projection gives
\[
    \nabla_\M f(x) = \frac{(-b^2,ab)}{a^2+b^2}
\]
and hence
\[
    \begin{aligned}
        \norm{\nabla_\M f(x)}
        &= \frac{|b|}{\sqrt{a^2+b^2}}
        = \frac{|x_2|^{p-1}}
          {(x_1^{2(p-1)} + |x_2|^{2(p-1)})^{1/2}} \\
        &= \Theta(|x_2|^{p-1})
    \end{aligned}
\]
near $\bar x$.
Meanwhile,
\[
    x_1 - r = -\frac{|x_2|^p}{p r^{p-1}} + O(|x_2|^{2p})
\]
is $o(|x_2|)$ as $p>1$, so the displacement along $x_1$ is of higher order
and
\[
    \dist(x,S) = \sqrt{(x_1-r)^2 + x_2^2} = \Theta(|x_2|).
\]
This computation establishes a GEB on $(\M,d)$ around $\bar x$
with the gauge $\phi(t) = t^{p-1}$, and \cref{thm:eb-equiv-geometric}
or \cref{thm:eb-equiv-metric} transfers it to the GEB on $(\R^2,d)$
with the same gauge.

\subsection{Recovering the \texorpdfstring{$\ell_1$}{l1} Case}
\label{sec:l1-recovery}
\looseness=-1
We now assume that the objective function $f$ has a regularized structure
of $f = g + h$,
with the following structural assumptions.
\begin{assumption}
	\label{ass:structural-assumption}
$\nabla g$ is Lipschitz continuous with constant
$L \geq 0$ and $h$ is proper, closed and \emph{$\rho$-weakly
convex} for some $\rho \geq 0$ (meaning that $h + \rho \|\cdot\|^2/2$ is convex).
\end{assumption}
Some useful weakly convex regularizers include
folded-concave penalties such as SCAD and MCP \cite{liu2017folded}.
In this scenario, $f$ is subdifferentially regular,
so the slope GEB and the subdifferential GEB coincide by
\cref{prop:slope}.
We refer to both as the \emph{subdifferential GEB} in this section.
The following definition is an alternative form of error bound (proximal EB)
commonly used in the optimization-algorithms literature, and with the
linear gauge, it is widely known as the Luo--Tseng error bound
\cite{luo1993error}.

\begin{definition}
    \label{def:proximal-eb}
	Suppose that $f = g + h$ satisfies
	\cref{ass:structural-assumption} and let $S \subset \crit f$.
	We say that $f$ satisfies the proximal generalized error bound
	around $\bar x \in S$ on $(\mathbb{R}^n,d)$ if there exist
	constants $\hat \epsilon,\hat \theta, \gamma > 0$ and a gauge
	function $\hat\phi$ such that
        \begin{equation*}
           \hat\phi(\dist(x, S)) \le \hat\theta \| x -
		   \prox_{\gamma h}(x -  \gamma \nabla g(x)) \|,\quad
	\forall x \in B_{\hat \epsilon}(\bar x),
        \end{equation*}
		where
		\begin{equation}
			\prox_{\gamma h}(x) \coloneqq  \argmin_{y \in \mathbb{R}^n}\,
			h(y) + \frac{1}{2\gamma} \| y - x \|^2.
			\label{eq:prox}
		\end{equation}
\end{definition}

Unlike \cref{def:subdiff-eb,def:subdiff-eb-m}, no sublevel restriction
appears above.
From \cref{ass:structural-assumption}, $f$ is weakly convex, so the
restriction is unnecessary in this setting by \cref{rem:sublevel}.
We may therefore also drop it from
\cref{def:subdiff-eb,def:subdiff-eb-m} here, so that both error bounds are
on the same region.

We now show that, under \cref{ass:structural-assumption}, the
subdifferential GEB and the proximal GEB are equivalent.
A similar result for convex $g$ and $h$ with $\phi(t) = t$ was shown
in \cite[Thm.s 3.4 and 3.5]{drusvyatskiy2018error}.

\begin{lemma}[Subdifferential EB and proximal EB]
    \label{lem:eb-equiv-proximal}
	Let \cref{ass:structural-assumption} hold, $S \subseteq \crit
	f$ with $\bar x \in S$, and $\gamma \rho < 1$.
	Then proximal GEB (\cref{def:proximal-eb}) with gauge $\phi$
	implies subdifferential GEB (\cref{def:subdiff-eb}) with the same
	gauge.
	Conversely, subdifferential GEB with gauge $\phi$ and constant
	$\theta$ implies proximal GEB with
	\begin{equation}
		\hat\phi(t) = \min\bigl\{ t,\ \phi(t/2) \bigr\},
		\qquad \hat\theta = \max\{2,c\},\qquad
		c \coloneqq \theta(1+\gamma L)/\gamma.
		\label{eq:degraded-gauge}
	\end{equation}
\end{lemma}
\begin{proof}
	Since $h$ is $\rho$-weakly convex and $\gamma\rho<1$,
	$\prox_{\gamma h}$ is single valued and
	$(1-\rho\gamma)^{-1}$-Lipschitz continuous
	\cite[Prop.~6.3]{bauschke2021generalized}; see also
	\cite[Prop.~12.19]{rockafellar2009variational}. That is,
	\begin{equation}
		\norm{z_1 - z_2} \geq (1 - \rho \gamma)
		\norm{\prox_{\gamma h}(z_1) - \prox_{\gamma h}(z_2)},\quad \forall z_1, z_2.
		\label{eq:proxexpand}
	\end{equation}

	Now assume subdifferential GEB holds with constants
	$\epsilon,\theta>0$ and gauge $\phi$, and let
	$\hat x \coloneqq \prox_{\gamma h}(x - \gamma\nabla g(x))$.
	Since $\bar x \in \crit f$ gives $-\nabla g(\bar x) \in \partial
	h(\bar x)$, which is the optimality condition of \cref{eq:prox} at
	$\bar x$ with the input $\bar x - \gamma \nabla g(\bar x)$, we have
	$\prox_{\gamma h}(\bar x - \gamma\nabla g(\bar x)) = \bar x$, and thus
	\cref{eq:proxexpand} gives
	\begin{equation*}
		\| \hat x - \bar x \| \le \frac{1}{1-\rho \gamma}
		\| x - \bar x - \gamma\nabla g(x) + \gamma\nabla g(\bar x) \|
		\le \frac{1+\gamma L}{1-\rho \gamma} \| x - \bar x \|.
	\end{equation*}
	Hence, for any $x \in B_{\hat\epsilon}(\bar x)$ with
	$\hat \epsilon \coloneqq \epsilon(1-\rho \gamma)/(1+\gamma L)$,
	the point $\hat x$ lies in the region where the subdifferential GEB
	applies.
	The optimality condition of \cref{eq:prox} at $\hat x$ reads
	$\gamma^{-1}(x - \hat x) - \nabla g(x) \in \partial h(\hat x)$, so that
	$\gamma^{-1}(x - \hat x) + \nabla g(\hat x) - \nabla g(x) \in \nabla
	g(\hat x) + \partial h(\hat x) = \partial f(\hat x)$, and therefore
	\begin{equation}
		\phi\bigl(\dist(\hat x, S)\bigr)
		\le \theta \dist\bigl(0, \partial f(\hat x)\bigr)
		\le \gamma^{-1}\theta(1+\gamma L)\| x - \hat x \|
		= c \| x - \hat x \|.
		\label{eq:gauge-at-xhat}
	\end{equation}
	It remains to transfer this estimate from $\hat x$ back to $x$.
	For a nonlinear gauge, the two terms in
	$\dist(x,S) \le \| x - \hat x \| + \dist(\hat x, S) \le 2\max\{\| x
	- \hat x \|, \dist(\hat x, S)\}$
	cannot be combined under $\phi$, so we keep the maximum.
	If it is attained by the first term, then $\dist(x,S) \le 2\| x -
	\hat x \|$; otherwise $\dist(x,S)/2 \le \dist(\hat x,S)$, and
	\cref{eq:gauge-at-xhat} together with the monotonicity of $\phi$
	gives $\phi(\dist(x,S)/2) \le c \| x - \hat x \|$.
	In both cases, $\hat\phi(\dist(x,S)) \le \max\{2,c\}\| x - \hat x
	\|$ with $\hat\phi$ defined in \cref{eq:degraded-gauge}, which is
	again a gauge function as the minimum of two gauge functions.
	This is exactly the proximal GEB.

    The converse result also holds. Assume proximal GEB
	holds with constants $\hat\epsilon,\hat\theta>0$ and gauge
	$\hat\phi$.
	Noting that
	$\operatorname{prox}_{\gamma h}(x + \gamma v - \gamma\nabla g(x)) = x$ for any $v
	\in \nabla g(x) + \partial h(x)$ and using \cref{eq:proxexpand}, we have
    \begin{align*}
    \| x - \hat x \| = \| \operatorname{prox}_{\gamma h}(x + \gamma v
	- \gamma\nabla g(x)) - \operatorname{prox}_{\gamma h}(x -
	\gamma\nabla g(x)) \| \le \gamma (1 - \rho \gamma)^{-1}\| v \|.
    \end{align*}
  Since $v$ is arbitrary,  we have $(1 - \rho \gamma)\| x - \hat x \| \le \gamma \operatorname{dist}(0, \partial f(x))$,
  and 
  hence,
  \begin{equation*}
	  \hat\phi(\dist(x,S))\leq \hat \theta \|x-\hat x\|\leq
	  \hat\theta \gamma(1 - \rho \gamma)^{-1} \dist(0,\partial
	  f(x)),\quad \forall x \in B_{\hat\epsilon}(\bar x).
  \end{equation*}
  This is exactly the subdifferential GEB.
\end{proof}
This confirms the claim in \cref{sec:eb}: with a general gauge,
\cref{def:proximal-eb} generalizes the Luo--Tseng error bound (which
corresponds to $\phi(t) = t$), and \cref{lem:eb-equiv-proximal} shows
its relation to \cref{def:subdiff-eb}.

As the error bounds concern only a neighborhood of $\bar x$ and
weak convexity is essentially the globalization of prox-regularity,
the assumption on $h$ in \cref{lem:eb-equiv-proximal} can be relaxed
to prox-regularity at the single point $\bar x$, with proper
modifications (we would need to add back the sublevel
restriction).
Likewise, the Lipschitz continuity of $\nabla g$ can be weakened to
H\"older continuity with exponent $\alpha \in (0,1]$, at the cost of
raising the power of $\phi(t/2)$ in \cref{eq:degraded-gauge} to $1/\alpha$
and hence a less sharp gauge.

In the introduction, we recalled that the equivalence between EB on
$(\mathbb{R}^n,d)$ and EB on $(\mathcal{M},d)$ for
$\ell_1$-regularization was shown in \cite{wu2025resolution} (for $S =
\argmin f$). Next, we
show that our main result
(\cref{thm:eb-equiv-geometric,thm:eb-equiv-metric}) can be applied to
$\ell_1$-regularization and recover this result with more generality
on $S$.

First, the $\ell_1$-norm is $\C^1$-partly smooth at any point $\bar x$
relative to the manifold $\mathcal{M} = \{ x : x_i = 0, \forall i \in
I \}$ where $I = \{ i : \bar x_i = 0 \}$, see for example
\cite[Ex.~1.2.5]{hare2011identifying}. %
Moreover, since $g$ is $\C^{1,1}$ by \cref{ass:structural-assumption},
the sum $f = g + h$ is also $\C^1$-partly smooth at $\bar x$ relative
to the same manifold $\mathcal{M}$.
As $h = \lambda\norm{\cdot}_1$ is convex, we have $\rho = 0$, so
\cref{lem:eb-equiv-proximal} applies to every $\gamma>0$ and we may simply
take $\gamma=1$.
Moreover, if the subdifferential EB holds with the linear gauge
$\phi(t) = t$, \cref{lem:eb-equiv-proximal} gives the proximal EB with
the same linear gauge.
Since \cref{thm:eb-equiv-geometric,thm:eb-equiv-metric} preserve the gauge as well,
they imply that if
\cref{ass:isocost,eq:nondegeneracy} hold at $\bar x$,
then proximal EB (\cref{def:proximal-eb}) with $\gamma = 1$ is
equivalent to EB on $\M$ (\cref{def:subdiff-eb-m}), both with a linear
gauge.
The setting $S = \argmin f$ in \cite[Thm.~4.3]{wu2025resolution}
makes \cref{ass:isocost} automatic, so the above recovers it (whose
EB definition likewise carries no sublevel restriction), and extends
to any $S \subseteq \crit f$ satisfying \cref{ass:isocost}.

\section{Conclusion}
\label{sec:conclusion}
We have studied the transfer of local error bounds between an ambient
Euclidean space and an identifiable $\C^1$ manifold, showing from both
a geometric and a metric perspective that an error bound in the ambient
space is equivalent to that on the manifold with the same gauge, so
restriction to the active manifold does not sharpen the error bound.
To establish this equivalence, we have developed tools to obtain a uniform
linear-growth estimate for nearby points off the manifold without
resorting to nearest-point projections that might not be available for
$\C^1$ manifolds.
We have further applied the developed tools to show ambient-to-manifold
transfer of the Kurdyka--\L ojasiewicz property under the same setting,
and established that the generalized error bound with a wide range of
gauge functions also implies this property, with an explicit conversion
from the gauge to its desingularizer.
Finally, for regularized optimization, we have related the generalized
error bound to its proximal form, of which the Luo--Tseng error bound
is the linear-gauge special case.
A possible direction for future work is to investigate the implications
of generalized error bounds with nonlinear gauges for local
convergence rates of active-set and manifold-based algorithms.

\backmatter
\bmhead{Acknowledgments}
\ifdefined\arxiv
The research of Hanju Wu and Yue Xie is funded by Hong Kong RGC
General Research Fund (Project number: 17300824), and Guangdong Province
Fundamental and Applied Fundamental Research Regional Joint Fund (Project
number: 2022B1515130009).
Ching-pei's research is supported in part by JSPS KAKENHI (24K20845),
JST PRESTO (JPMJPR25KC), and JST BOOST (JPMJBY25F7).
\fi
The authors thank Adrian Lewis and Tonghua Tian for their
correspondence on \cite[Thm.~5.12]{lewis2024identifiability}.
We also thank Bruno Louren{\c{c}}o and Man-Chung Yue for helpful
discussions on the improper integral in \cref{prop:eb-kl}. Limited AI assistance was used in generating examples in Sec~\ref{sec:lp-example} and Appendix, and the authors assume responsibility for all content.

\bibliography{sn-bibliography}

@article{liu2017folded,
  title={Folded concave penalized sparse linear regression: sparsity, statistical performance, and algorithmic theory for local solutions},
  author={Liu, Hongcheng and Yao, Tao and Li, Runze and Ye, Yinyu},
  journal={Mathematical Programming},
  volume={166},
  number={1},
  pages={207--240},
  year={2017},
  publisher={Springer}
}

@incollection{hare2011identifying,
  title={Identifying active manifolds in regularization problems},
  author={Hare, Warren L.},
  booktitle={Fixed-Point Algorithms for Inverse Problems in Science and Engineering},
  pages={261--271},
  year={2011},
  publisher={Springer},
  address={New York, NY}
}

@incollection{lee2013smooth,
  title={Smooth manifolds},
  author={Lee, John M.},
  booktitle={Introduction to smooth manifolds},
  pages={77--97},
  year={2013},
  publisher={Springer},
  address={New York, NY}
}

@article{liang2017activity,
  title={Activity identification and local linear convergence of forward--backward-type methods},
  author={Liang, Jingwei and Fadili, Jalal and Peyr{\'e}, Gabriel},
  journal={SIAM Journal on Optimization},
  volume={27},
  number={1},
  pages={408--437},
  year={2017},
  publisher={SIAM}
}

@article{daniilidis2006geometrical,
  title={Geometrical interpretation of the predictor-corrector type algorithms in structured optimization problems},
  author={Daniilidis, Aris and Hare, Warren L. and Malick, J{\'e}r{\^o}me},
  journal={Optimization},
  volume={55},
  number={5-6},
  pages={481--503},
  year={2006},
  publisher={Taylor \& Francis}
}

@article{lewis2025complexity,
  title={The complexity of first-order optimization methods from a
	  metric perspective},
  author={Lewis, Adrian S. and Tian, Tonghua},
  journal={Mathematical Programming},
  volume={212},
  number={1-2},
  pages={49--78},
  year={2025},
  publisher={Springer}
}

@article{miller2005newton,
  title={{N}ewton methods for nonsmooth convex minimization:
	  connections among {U}-Lagrangian, {R}iemannian {N}ewton and
	  {SQP} methods},
  author={Miller, Scott A. and Malick, J{\'e}r{\^o}me},
  journal={Mathematical Programming},
  volume={104},
  number={2},
  pages={609--633},
  year={2005},
  publisher={Springer}
}

@article{drusvyatskiy2021nonsmooth,
  title={Nonsmooth optimization using {T}aylor-like models: error bounds, convergence, and termination criteria},
  author={Drusvyatskiy, Dmitriy and Ioffe, Alexander D. and Lewis, Adrian S.},
  journal={Mathematical Programming},
  volume={185},
  number={1},
  pages={357--383},
  year={2021},
  publisher={Springer}
}

@article{rebjock2025fast,
  title={Fast convergence to non-isolated minima: four equivalent conditions for $C^2$ functions},
  author={Rebjock, Quentin and Boumal, Nicolas},
  journal={Mathematical Programming},
  volume={213},
  number={1-2},
  pages={151--199},
  year={2025},
  publisher={Springer}
}

@misc{wu2025resolution,
  title={On the resolution of $\ell_1$-norm minimization via a two-metric adaptive projection method},
  author={Wu, Hanju and Xie, Yue},
  eprint={2504.12260},
  note={\href{https://arxiv.org/abs/2504.12260}{arXiv:2504.12260}},
  year={2025}
}

@article{hu2025analysis,
  title={On the analysis of semismooth {N}ewton-type methods for composite optimization},
  author={Hu, Jiang and Tian, Tonghua and Pan, Shaohua and Wen, Zaiwen},
  journal={Journal of Scientific Computing},
  volume={103},
  number={2},
  pages={59},
  year={2025},
  publisher={Springer}
}

@article{lewis2013partial,
  title={Partial smoothness, tilt stability, and generalized {H}essians},
  author={Lewis, Adrian S. and Zhang, Shanshan},
  journal={SIAM Journal on Optimization},
  volume={23},
  number={1},
  pages={74--94},
  year={2013},
  publisher={SIAM}
}

@article{lewis2002active,
  title={Active sets, nonsmoothness, and sensitivity},
  author={Lewis, Adrian S.},
  journal={SIAM Journal on Optimization},
  volume={13},
  number={3},
  pages={702--725},
  year={2002},
  publisher={SIAM}
}

@misc{drusvyatskiy2012optimalityidentifiabilitysensitivity,
      title={Optimality, identifiability, and sensitivity}, 
      author={Dmitriy Drusvyatskiy and Adrian S. Lewis},
      year={2012},
      eprint={1207.6628},
      note={\href{https://arxiv.org/abs/1207.6628}{arXiv:1207.6628}},
}

@phdthesis{drusvyatskiy2013slope,
  title={Slope and geometry in variational mathematics},
  author={Drusvyatskiy, Dmitriy},
  year={2013},
  school={Cornell University}
}

@book{ioffe2017variational,
  title={Variational Analysis of Regular Mappings: Theory and Applications},
  author={Ioffe, Alexander D.},
  series={Springer Monogr. Math.},
  year={2017},
  publisher={Springer},
  address={Cham}
}

@article{li2018calculus,
  title={Calculus of the exponent of {K}urdyka--{{\L}}ojasiewicz inequality and its applications to linear convergence of first-order methods},
  author={Li, Guoyin and Pong, Ting Kei},
  journal={Foundations of Computational Mathematics},
  volume={18},
  number={5},
  pages={1199--1232},
  year={2018},
  publisher={Springer}
}

@Book{boumal2023intromanifolds,
  title     = {An introduction to optimization on smooth manifolds},
  author    = {Boumal, Nicolas},
  publisher = {Cambridge University Press},
  year      = {2023},
  address = {UK}
}

@article{wright1993identifiable,
  title={Identifiable surfaces in constrained optimization},
  author={Wright, Stephen J.},
  journal={SIAM Journal on Control and Optimization},
  volume={31},
  number={4},
  pages={1063--1079},
  year={1993},
  publisher={SIAM}
}

@article{bertsekas1982projected,
  title={Projected {N}ewton methods for optimization problems with simple constraints},
  author={Bertsekas, Dimitri P.},
  journal={SIAM Journal on Control and Optimization},
  volume={20},
  number={2},
  pages={221--246},
  year={1982},
  publisher={SIAM}
}

@book{rockafellar2009variational,
  title={Variational analysis},
  author={Rockafellar, Ralph Tyrrell and Wets, Roger J.-B.},
  volume={317},
  year={2009},
  publisher={Springer Science \& Business Media},
  address={Berlin, Heidelberg}
}

@article{yue2019family,
  title={A family of inexact {SQA} methods for non-smooth convex minimization with provable convergence guarantees based on the {L}uo--{T}seng error bound property},
  author={Yue, Man-Chung and Zhou, Zirui and So, Anthony Man-Cho},
  journal={Mathematical Programming},
  volume={174},
  number={1},
  pages={327--358},
  year={2019},
  publisher={Springer}
}

@article{gafni1984two,
  title={Two-metric projection methods for constrained optimization},
  author={Gafni, Eli M. and Bertsekas, Dimitri P.},
  journal={SIAM Journal on Control and Optimization},
  volume={22},
  number={6},
  pages={936--964},
  year={1984},
  publisher={SIAM}
}

@article{zhou2017unified,
  title={A unified approach to error bounds for structured convex optimization problems},
  author={Zhou, Zirui and So, Anthony Man-Cho},
  journal={Mathematical Programming},
  volume={165},
  pages={689--728},
  year={2017},
  publisher={Springer}
}

@article{hou2013linear,
  title={On the linear convergence of the proximal gradient method for trace norm regularization},
  author={Hou, Ke and Zhou, Zirui and So, Anthony Man-Cho and Luo, Zhi-Quan},
  journal={Advances in Neural Information Processing Systems},
  volume={26},
  year={2013}
}

@article{luo1993error,
  title={Error bounds and convergence analysis of feasible descent methods: a general approach},
  author={Luo, Zhi-Quan and Tseng, Paul},
  journal={Annals of Operations Research},
  volume={46},
  number={1},
  pages={157--178},
  year={1993},
  publisher={Springer}
}

@article{hare2004identifying,
  title={Identifying active constraints via partial smoothness and prox-regularity},
  author={Hare, Warren L. and Lewis, Adrian S.},
  journal={Journal of Convex Analysis},
  volume={11},
  number={2},
  pages={251--266},
  year={2004},
  publisher={HELDERMANN VERLAG LANGER GRABEN 17, 32657 LEMGO, GERMANY}
}

@article{tseng2009coordinate,
  title={A coordinate gradient descent method for nonsmooth separable minimization},
  author={Tseng, Paul and Yun, Sangwoon},
  journal={Mathematical Programming},
  volume={117},
  pages={387--423},
  year={2009},
  publisher={Springer}
}

@article{lewis2024identifiability,
  title={Identifiability, the {KL} property in metric spaces, and subgradient curves},
  author={Lewis, Adrian S. and Tian, Tonghua},
  journal={Foundations of Computational Mathematics},
  volume={25},
  number={3},
  pages={905--942},
  year={2025},
  publisher={Springer}
}

@article{lewis2022partial,
  title={Partial smoothness and constant rank},
  author={Lewis, Adrian S. and Liang, Jingwei and Tian, Tonghua},
  journal={SIAM Journal on Optimization},
  volume={32},
  number={1},
  pages={276--291},
  year={2022},
  publisher={SIAM}
}

@article{davis2024stochastic,
  title={Stochastic algorithms with geometric step decay converge linearly on sharp functions},
  author={Davis, Damek and Drusvyatskiy, Dmitriy and Charisopoulos, Vasileios},
  journal={Mathematical Programming},
  volume={207},
  number={1},
  pages={145--190},
  year={2024},
  publisher={Springer}
}

@article{li2024generalizedsubregularity,
  title={Generalized Metric Subregularity with Applications to High-Order Regularized {N}ewton Methods},
  author={Li, Guoyin and Mordukhovich, Boris S. and Zhu, Jiangxing},
  journal={Mathematics of Operations Research},
  year={2026},
  note={Articles in Advance},
  doi={10.1287/moor.2024.0570},
}

@article{drusvyatskiy2018error,
  title={Error bounds, quadratic growth, and linear convergence of proximal methods},
  author={Drusvyatskiy, Dmitriy and Lewis, Adrian S.},
  journal={Mathematics of Operations Research},
  volume={43},
  number={3},
  pages={919--948},
  year={2018},
  publisher={INFORMS}
}

@article{tseng2010approximation,
  title={Approximation accuracy, gradient methods, and error bound for structured convex optimization},
  author={Tseng, Paul},
  journal={Mathematical Programming},
  volume={125},
  number={2},
  pages={263--295},
  year={2010},
  publisher={Springer}
}

@article{drusvyatskiy2014optimality,
  title={Optimality, identifiability, and sensitivity},
  author={Drusvyatskiy, Dmitriy and Lewis, Adrian S.},
  journal={Mathematical Programming},
  volume={147},
  number={1},
  pages={467--498},
  year={2014},
  publisher={Springer}
}

@article{bauschke2021generalized,
	title={Generalized monotone operators and their averaged resolvents},
	author={Bauschke, Heinz H. and Moursi, Walaa M. and Wang, Xianfu},
	journal={Mathematical Programming},
	volume={189},
	pages={55--74},
	year={2021}
}

@article{AttBS13a,
	title={Convergence of descent methods for semi-algebraic and tame
		problems: proximal algorithms, forward--backward splitting,
		and regularized {G}auss--{S}eidel methods},
	author={Attouch, Hedy and Bolte, J{\'e}r{\^o}me and Svaiter, Benar Fux},
	journal={Mathematical Programming},
	volume={137},
	number={1-2},
	pages={91--129},
	year={2013},
	publisher={Springer}
}

@article{bolte2010characterizations,
  title={Characterizations of {{\L}}ojasiewicz inequalities: subgradient flows, talweg, convexity},
  author={Bolte, J{\'e}r{\^o}me and Daniilidis, Aris and Ley, Olivier and Mazet, Laurent},
  journal={Transactions of the American Mathematical Society},
  volume={362},
  number={6},
  pages={3319--3363},
  year={2010}
}

@article{wen2017linear,
	title={Linear convergence of proximal gradient algorithm with extrapolation for a class of nonconvex nonsmooth minimization problems},
	author={Wen, Bo and Chen, Xiaojun and Pong, Ting Kei},
	journal={SIAM Journal on Optimization},
	volume={27},
	number={1},
	pages={124--145},
	year={2017}
}

@article{ioffe2009invitation,
  title={An invitation to tame optimization},
  author={Ioffe, Alexander D.},
  journal={SIAM Journal on Optimization},
  volume={19},
  number={4},
  pages={1894--1917},
  year={2009},
}

@article{bolte2007clarke,
  title={{C}larke subgradients of stratifiable functions},
  author={Bolte, J{\'e}r{\^o}me and Daniilidis, Aris and Lewis, Adrian S. and Shiota, Masahiro},
  journal={SIAM Journal on Optimization},
  volume={18},
  number={2},
  pages={556--572},
  year={2007},
}

@article{bolte2014proximal,
  title={Proximal alternating linearized minimization for nonconvex and nonsmooth problems},
  author={Bolte, J{\'e}r{\^o}me and Sabach, Shoham and Teboulle, Marc},
  journal={Mathematical Programming},
  volume={146},
  number={1},
  pages={459--494},
  year={2014},
  publisher={Springer}
}

@article{liu2026concrete,
  title={Concrete convergence rates for common fixed point problems under {Karamata} regularity},
  author={Liu, Tianxiang and Louren{\c{c}}o, Bruno F.},
  journal={Mathematical Programming},
  year={2026},
  note={Online first},
  doi={10.1007/s10107-026-02332-6},
}

@article{wright2012accelerated,
  author    = {Wright, Stephen J.},
  title     = {Accelerated Block-Coordinate Relaxation for Regularized Optimization},
  journal   = {SIAM Journal on Optimization},
  volume    = {22},
  number    = {1},
  pages     = {159--186},
  year      = {2012},
  publisher = {Society for Industrial and Applied Mathematics},
}

@article{Lee2023a,
	title={Accelerating Inexact Successive Quadratic Approximation for Regularized
Optimization Through Manifold Identification},
	author = {Lee, Ching-pei},
	year={2023},
	volume={201},
	number={1},
	pages={599--633},
	journal={Mathematical Programming},
}

@misc{diaz2026active,
  title={Active set identification and rapid convergence for degenerate primal-dual problems},
  author={D{\'\i}az, Mateo and Lehmann, Pedro Izquierdo and Lu, Haihao and Yang, Jinwen},
  eprint={2602.10436},
  note={\href{https://arxiv.org/abs/2602.10436}{arXiv:2602.10436}},
  year={2026}
}

\begin{appendices}
\renewcommand{\theequation}{\thesection.\arabic{equation}}
\renewcommand{\theHsection}{appendix.\Alph{section}}

\section{A Convex Counterexample in Euclidean Space}
\label{sec:counterexample}

\label{sec:counterexample-euclidean}

Let $X=\R^2$ with the Euclidean metric, and define
    \[f(u,v)\coloneqq|v|,\qquad \bar x\coloneqq (0,0), \qquad
    \M\coloneqq(\R\times\{0\})\cup\{(u,u^2):u\in\R\}.
\]
The function $f$ is convex and Lipschitz continuous, and $\M$ is
closed, connected, and semialgebraic.
Since $|\nabla f|(\bar x)=0$ and
$|\nabla f|(u,v)=1$ whenever $v\ne0$, the function $f$ has the
KL property at $\bar x$ with desingularizer $\psi(t)=t$.
Moreover, every point outside $\M$ has $v\ne0$, so we get
$\M$ is identifiable at $\bar x$ for $f$ with modulus $\eta = 1$.

For $u\ne0$, $\M$ coincides with the parabola
near $x_u\coloneqq (u,u^2)$.
The parametrization $u\mapsto(u,u^2)$ has derivative $(1,2u)$,
and $f(x_u)=u^2$.
Consequently, \cref{lem:slope-grad} gives
\[
    |\nabla(f|_{\M})|(x_u)
    =2|u|/\sqrt{1+4u^2}
    \longrightarrow0
    \qquad\text{as }u\to0.
\]
Since $x_u\to\bar x$ and $f(x_u)\to f(\bar x)=0$ from above,
the restriction $f|_{\M}$ fails the KL inequality with
$\psi(t)=t$.
This disproves the ambient-to-restricted implication in
\cite[Thm.~5.12]{lewis2024identifiability} even in a
Euclidean setting.

\section{Proofs of Results in \crefhead{sec:preliminary}{Section}}
\label{proofs}

\subsection{Proof of \crefhead{lem:local-graph}{Lemma}}
    Let $m=\dim \mathcal{M}$ and choose a $\C^p$ local defining map
    $F:U_0\to\mathbb{R}^{n-m}$ around $\bar x$ such that
    $\mathcal{M}\cap U_0=F^{-1}(0)$ and $\nabla F(\bar x)$ has full rank.
    Set $T\coloneqq T_{\bar x}\mathcal{M}=\ker\nabla F(\bar x)$ and $N\coloneqq N_{\bar x}\mathcal{M}=T^\perp$.
    Consider
    \[
        G(u,w)\coloneqq F(\bar x+u+w),\qquad u\in T,\quad w\in N.
    \]
    The derivative of $G$ with respect to $w$ at $(0,0)$ is the linear map
    \[
        \mathrm{D}_wG(0,0)w=\nabla F(\bar x)w,\qquad w\in N.
    \]
    This map from $N$ to $\mathbb{R}^{n-m}$ is an isomorphism: it is injective because if $w\in N$ and $\nabla F(\bar x)w=0$, then $w\in T\cap N=\{0\}$, and the dimensions of $N$ and $\mathbb{R}^{n-m}$ agree.
    By the implicit function theorem, there exist $\rho>0$, a neighborhood $O_N$ of $0$ in $N$, and a unique $\C^p$ map $h:T\cap B_\rho(0)\to O_N$ such that $h(0)=0$ and
    \[
        F(\bar x+u+w)=0
        \quad\Longleftrightarrow\quad
        w=h(u)
    \]
    for $(u,w)$ sufficiently small.  Hence $\mathcal{M}$ is locally represented as the stated graph.
    Differentiating the identity $F(\bar x+u+h(u))=0$ at $u=0$ gives
    \[
        \nabla F(\bar x)\bigl(u+\nabla h(0)u\bigr)=0,
        \qquad \forall u\in T.
    \]
    Since $u\in\ker\nabla F(\bar x)$, we have $\nabla F(\bar x)\nabla h(0)u=0$.
    But $\nabla h(0)u\in N$ and $\nabla F(\bar x)|_N$ is injective, so $\nabla h(0)u=0$ for all $u\in T$, and
    therefore $\nabla h(0)=0$.  The final assertion follows by decomposing $y-\bar x$ according to $\mathbb{R}^n=T\oplus N$ and shrinking the neighborhood so that the tangential component lies in $B_\rho(0)$.

\subsection{Proof of \crefhead{prop:subdiff-inclusion}{Prop.}}
    Fix $g\in\hat\partial f(x)$ and $\xi\in T_x\mathcal M$.
    Choose a $\C^1$ curve
    $\gamma:(-\epsilon,\epsilon)\to\mathcal M$ satisfying
    $\gamma(0)=x$ and $\dot\gamma(0)=\xi$.  
    The definition of a Fr\'echet subgradient gives
    \begin{equation*}
        f(\gamma(t))-f(x)
        \geq \langle g,\gamma(t)-x\rangle+o(\left\|\gamma(t)-x\right\|).
    \end{equation*}
    Since $f|_{\mathcal M}$ is $\C^1$ around $x$, dividing
    by $t>0$ and letting $t\downarrow0$ yields
    \begin{equation*}
        \mathrm D(f|_{\mathcal M})(x)[\xi]\geq\langle g,\xi\rangle,
    \end{equation*}
	while dividing both sides by $t<0$ and letting $t\uparrow0$ yields the
	reverse inequality.
	Therefore,
    \begin{equation*}
        \mathrm D(f|_{\mathcal M})(x)[\xi]
        =\langle g,\xi\rangle
        =\langle P_{T_x\mathcal M}(g),\xi\rangle.
    \end{equation*}
    By the definition of the Riemannian gradient, it follows that
    \begin{equation*}
        \nabla_{\mathcal M}f(x)=P_{T_x\mathcal M}(g).
    \end{equation*}
    Hence, any two elements of $\hat\partial f(x)$ have the same
    projection onto $T_x\mathcal M$, proving
    \cref{eq:sharpness-easy-inclusion}.

\subsection{Proof of \crefhead{lem:slope-grad}{Lemma}}
Let
\[
    T\coloneqq T_x\mathcal{M},
    \qquad
    N\coloneqq N_x\mathcal{M},
    \qquad
    g\coloneqq \nabla_{\mathcal{M}} f(x)\in T .
\]
By \cref{lem:local-graph}, after shrinking the neighborhood of $x$ if necessary, there is a $\C^1$ map
\[
    h:T\cap B_\rho(0)\to N
\]
such that $h(0)=0$, $\nabla h(0)=0$, and
\[
    \Phi(u)\coloneqq x+u+h(u), \qquad u\in T\cap B_\rho(0),
\]
parametrizes $\mathcal{M}$ near $x$.
Since $h(u)=o(\|u\|)$, we have
\begin{equation}
    \label{eq:local-graph-distance}
    \|\Phi(u)-x\|=\|u+h(u)\|=\|u\|+o(\|u\|).
\end{equation}
Define
\[
    F:T\cap B_\rho(0)\to\mathbb{R},
    \qquad
    F(u)\coloneqq f(\Phi(u)).
\]
Since $f|_{\mathcal{M}}$ and $\Phi$ are $\C^1$, the function $F$ is $\C^1$ near $0$ in the vector space $T$.
Moreover,
\[
    \mathrm{D}\Phi(0)[u]=u+\nabla h(0)u=u,
    \qquad u\in T.
\]
By the chain rule and the definition of the Riemannian gradient,
\[
    \mathrm{D}F(0)[u]
    =
    \mathrm{D}(f|_{\mathcal{M}})(x)\big[\mathrm{D}\Phi(0)[u]\big]
    =
    \mathrm{D}(f|_{\mathcal{M}})(x)[u]
    =
    \langle g,u\rangle,
    \qquad u\in T.
\]
Therefore the first-order Taylor expansion of $F$ at $0$ is
\begin{equation}
    \label{eq:local-graph-taylor}
    F(u)=F(0)+\langle g,u\rangle+o(\|u\|).
\end{equation}

We first prove the upper bound.
If $x$ is a local minimizer of $f|_{\mathcal{M}}$, then
\[
    |\nabla(f|_{\mathcal{M}})|(x)=0\leq \|g\|.
\]
Otherwise, let $y_r\in\mathcal{M}$ satisfy $y_r\to x$ and $y_r\ne x$.
Writing $y_r=\Phi(u_r)$ with $u_r\to 0$ and using \crefrange{eq:local-graph-distance}{eq:local-graph-taylor}, we obtain
\[
    \frac{f(x)-f(y_r)}{\|x-y_r\|}
    =
    \frac{-\langle g,u_r\rangle+o(\|u_r\|)}
         {\|u_r\|+o(\|u_r\|)}
    \leq \frac{\|g\|+o(1)}{1+o(1)}.
\]
Taking the limsup over all such sequences yields
\[
    |\nabla(f|_{\mathcal{M}})|(x)\leq \|g\|.
\]

For the reverse inequality, if $g=0$, the upper bound and nonnegativity of the slope already imply equality.
Assume $g\ne 0$ and set $v=-g/\|g\|\in T$.
For $t>0$ small, let $y_t=\Phi(tv)\in\mathcal{M}$.
Using \crefrange{eq:local-graph-distance}{eq:local-graph-taylor} again, we get
\[
    \frac{f(x)-f(y_t)}{\|x-y_t\|}
    =
    \frac{t\|g\|+o(t)}{t+o(t)}
    \to \|g\|
    \qquad (t\downarrow 0).
\]
Hence $|\nabla(f|_{\mathcal{M}})|(x)\geq \|g\|$.
The two inequalities prove the claim.

\section{Proofs of Results in \crefhead{sec:eb}{Section}}
\label{proofs-vu}

\subsection{ Proof of \crefhead{lem:u-gradient-property}{Lemma}}
    Fix $x\in\mathcal M$ sufficiently close to $\bar x$.
    By \cref{prop:sharpness-stability},
    $\operatorname{par}\hat\partial f(x)=N_x\mathcal M$.
    Thus, for any $g,g'\in\hat\partial f(x)$,
    $g-g'\in N_x\mathcal M$, so that
    \[
        P_{T_x\mathcal M}(g-g')=0.
    \]
    Hence the tangential projection is independent of the choice of
    subgradient.  Let $p=P_{\operatorname{aff}\hat\partial f(x)}(0)$.
    The characterization of the Euclidean projection onto an affine
    space gives $p\perp N_x\mathcal M$, while
    $g-p\in N_x\mathcal M$ for every $g\in\hat\partial f(x)$.
    Consequently,
    \[
        P_{T_x\mathcal M}(g)=p
        =P_{\operatorname{aff}\hat\partial f(x)}(0).
    \]

	\subsection{Proof of \crefhead{lem:ri-subgradient}{Lemma}}
    By \cref{eq:Rgrad-Fsubdiff}, it suffices to show that
    $\nabla_{\mathcal M}f(x)\in\operatorname{ri}\hat\partial f(x)$
    for all $x\in\mathcal M$ near $\bar x$.
    Prox-regularity at $\bar x$ for $0$ gives
    $0\in\partial_P f(\bar x)\subseteq\hat\partial f(\bar x)$,
    so \cref{eq:Rgrad-Fsubdiff} implies
    $\nabla_{\mathcal M}f(\bar x)=0$.
    Since $f|_{\mathcal M}$ and $\mathcal M$ are $\C^1$-smooth,
    the Riemannian gradient is continuous, and hence
    \begin{equation}
        \label{eq:ri-subgradient-continuity}
        \nabla_{\mathcal M}f(x)\to0
        \qquad\text{as }x\to\bar x\text{ in }\mathcal M.
    \end{equation}

    For $x\in\mathcal M$ near $\bar x$, the set
    $\hat\partial f(x)$ is nonempty, closed, and convex.
    By \cref{prop:sharpness-stability,lem:u-gradient-property} and
    \cref{eq:Rgrad-Fsubdiff},
    \[
        \operatorname{aff}\hat\partial f(x)
        =\nabla_{\mathcal M}f(x)+N_x\mathcal M.
    \]
    If $N_{\bar x}\mathcal M=\{0\}$, this identity shows that
    $\hat\partial f(x)=\{\nabla_{\mathcal M}f(x)\}$ for all nearby
    $x\in\mathcal M$, and the conclusion follows.
    Otherwise, suppose for a contradiction that there are
    $x_k\to\bar x$ in $\mathcal M$ such that
    $\nabla_{\mathcal M}f(x_k)\notin
    \operatorname{ri}\hat\partial f(x_k)$.
    Applying convex separation in the normal space $N_{x_k}\mathcal M$
    to the translated set
    $\hat\partial f(x_k)-\nabla_{\mathcal M}f(x_k)$, we obtain unit
    vectors $w_k\in N_{x_k}\mathcal M$ with
    \begin{equation}
        \label{eq:ri-subgradient-separation}
        \langle w_k,v-\nabla_{\mathcal M}f(x_k)\rangle\leq0
        \qquad\forall v\in\hat\partial f(x_k).
    \end{equation}
    Passing to a subsequence and using continuity of the normal
    spaces, assume that $w_k\to w\in N_{\bar x}\mathcal M$
    with $\|w\|=1$.

    Sharpness and nondegeneracy at $\bar x$ give
    $\delta w\in\hat\partial f(\bar x)$ for some $\delta>0$.
    Choose $\delta$ small enough that $\delta w$ lies in the
    neighborhoods of $0$ where both inner semicontinuity and
    prox-regularity in \cref{def:partial-smoothness} apply.
    Inner semicontinuity provides $v_k\in\partial f(x_k)$ with
    $v_k\to\delta w$.
    Since $f(x_k)\to f(\bar x)$, prox-regularity at $\bar x$ for
    $0$ implies
    $v_k\in\partial_P f(x_k)\subseteq\hat\partial f(x_k)$ for all
    sufficiently large $k$.
    Applying \cref{eq:ri-subgradient-separation} to $v_k$ and using
    \cref{eq:ri-subgradient-continuity}, we obtain
    \[
        0\geq\lim_{k\to\infty}
        \langle w_k,v_k-\nabla_{\mathcal M}f(x_k)\rangle
        =\delta\|w\|^2=\delta>0,
    \]
    a contradiction.
\subsection{Proof of \crefhead{cor:u-gradient-projection}{Lemma}}
    By \cref{lem:u-gradient-property,lem:ri-subgradient} and
    \cref{eq:Rgrad-Fsubdiff},
    \[
        \nabla_{\mathcal M}f(x)
        =P_{\operatorname{aff}\hat\partial f(x)}(0)
        \in\hat\partial f(x).
    \]
    Thus the projection of $0$ onto the affine hull already belongs
    to $\hat\partial f(x)$, and is also its unique minimum-norm
    element.  The two asserted identities follow.

\end{appendices}

\end{document}